\documentclass[10pt,reqno]{amsart}

\usepackage{srcltx}
\usepackage{amsmath,amssymb,mathtools,amsfonts,amscd,cases,xcolor} \usepackage{hyperref}
\usepackage[capitalise]{cleveref}

\usepackage[top=2.7cm,bottom=2.7cm,left=2.5cm,
right=2.5cm]{geometry} \usepackage{mathrsfs} \usepackage{graphicx} \usepackage{upref}
\usepackage{dsfont}
\usepackage[utf8]{inputenc}
\usepackage{comment}

\usepackage[sort,nocompress]{cite}

\hypersetup{linkcolor=blue, colorlinks=true,citecolor = red}

\newtheorem{theorem}{Theorem}[section] 
\newtheorem{proposition}[theorem]{Proposition}
 \newtheorem{cor}[theorem]{Corollary} \newtheorem{remark}[theorem]{Remark} \newtheorem{defin}[theorem]{Definition}

\theoremstyle{definition}
 
\theoremstyle{remark}
\numberwithin{equation}{section}

\DeclareMathOperator*{\esssup}{ess\,sup}
 \newcommand{\mc}{\mathcal} \renewcommand{\leq}{\leqslant} \renewcommand{\geq}{\geqslant}
\newcommand{\Div}       {{\rm div}_x}

\newcommand{\F} {\mathcal{F}}
\renewcommand{\S} {\mathcal{S}}

\renewcommand{\d}{{\rm d} }

\newcommand{\R}{{\mathbb R}}

\newcommand{\veps}{\varepsilon}
\newcommand{\Q}{\mathbb{Q}}
\DeclareMathOperator{\ess}{\textnormal{ess}}
\DeclareMathOperator{\res}{\textnormal{res}}
\begin{document}

\title[Singular limits of fluid--rigid body system]{Boundary Layer Criterion for the 
Low Mach Number and Inviscid Limits in Compressible 
Fluid--Rigid Body Interaction}

\date{\today}

\author[K. Bhandari]{Kuntal Bhandari} \address{Institute of Mathematics, Czech Academy of Sciences, \v{Z}itn\'{a} 25, 11567 Praha 1, Czech Republic} \email{bhandari@math.cas.cz}

\author[\v{S}. Ne\v{c}asov\'{a}]{\v{S}\'{a}rka  Ne\v{c}asov\'{a}} \address{Institute of Mathematics, Czech Academy of Sciences, \v Zitn\' a 25, 11567 Praha 1, Czech Republic} \email{matus@math.cas.cz}
 
\author[A. Roy]{Arnab Roy}
\address{BCAM, Basque Center for Applied Mathematics, Mazarredo 14, E48009 Bilbao, Bizkaia, Spain}
\address{IKERBASQUE, Basque Foundation for Science, 
Plaza Euskadi 5, 48009 Bilbao, Bizkaia, Spain}
 \email{aroy@bcamath.org}

\begin{abstract}
We investigate the simultaneous low Mach number and vanishing viscosity limit for a three-dimensional compressible viscous fluid interacting with a freely moving rigid body. The fluid satisfies the compressible Navier--Stokes equations and the no-slip condition on the body surface, while the body evolves according to Newton’s laws driven by the fluid stresses. We establish convergence of finite-energy weak solutions to a strong solution of the corresponding incompressible Euler--rigid body system on its interval of existence. Assuming that the initial relative energy tends to zero, convergence holds under a Kato-type condition in a boundary layer whose thickness is proportional to the viscosity. This condition combines vanishing viscous dissipation with a weighted estimate on the fluid momentum relative to the rigid body motion. The proof uses the relative entropy method and a boundary layer corrector adapted to the fluid-rigid body coupling, providing a sufficient condition for convergence in relative energy as the Mach number and viscosity tend to zero simultaneously.
\end{abstract}

\maketitle


\vspace{.2cm} 


\tableofcontents

%%%%%%%%%%%%%%%%%%%%%%%%%%%%%%%%%%%%%%%%%%%%%%%%%%%%%%%%%%%%%%%%%%%%%%

\section{Introduction}
The simultaneous low Mach number and vanishing viscosity limit connects two distinct descriptions of fluid-rigid body interaction: a rigid body immersed in a compressible viscous fluid and a rigid body moving in an incompressible inviscid fluid. In both systems, the body evolves according to Newton’s laws under the action of the fluid stresses. The fluid equations and the interface conditions, however, change in the limit. The density approaches a constant and the fluid velocity becomes divergence-free, while the viscous no-slip condition gives way to an impermeability condition that allows tangential slip along the body surface. The question is whether solutions of the compressible Navier-Stokes-rigid body system converge to those of the incompressible Euler-rigid body system as the Mach number and viscosity tend to zero.

This passage involves two singular mechanisms. The low Mach number limit introduces the incompressibility constraint and may generate rapid acoustic oscillations, whereas the vanishing viscosity limit is sensitive to boundary layers associated with the mismatch of tangential boundary conditions. For fluids in fixed domains, these mechanisms have been studied extensively \cite{MAS5, MAS1, KM1, Feireisl-Novotny-INVISCID, MR277929, Kato-zero-viscosity}. In the presence of a freely moving rigid body, the boundary motion is itself an unknown, coupled to the fluid through the balance of linear and angular momentum. Establishing the limit therefore requires controlling the fluid and body motions together, including the behaviour of the fluid relative to the body in a thin neighbourhood of its surface.

The motion of rigid bodies in incompressible viscous fluids has been extensively studied. Previous contributions to the existence theory of weak solutions include \cite{CoJoTu99,DeEs99,GLSE,HOST,SER3,DEES2,MR4223872}. Global weak formulations accommodating possible collisions between bodies or contact with the boundary were subsequently developed in \cite{F3,SaStTu02,CN}.

For compressible viscous fluids, Boulakia and Guerrero \cite{BG} established the existence and uniqueness of strong solutions in an $L^2$-based Sobolev framework for small initial data, as long as no collision occurs. Hieber and Murata \cite{HiMu15} proved local well-posedness in an $L^p$ framework using maximal regularity methods based on $\mathcal{R}$-boundedness. In a related stabilization problem, Roy and Takahashi \cite{roy2019stabilization}
 obtained global strong solutions for small initial data and asymptotic stabilization of a rigid ball under a feedback force acting on the body.

For weak solutions, existence under no-slip boundary conditions was established up to collision in \cite{DEES2}, without a smallness assumption on the initial data. Feireisl \cite{Feireisl-compressible-Rigid} subsequently constructed global weak solutions in a formulation that accommodates possible collisions. Weak–strong uniqueness for the compressible fluid–rigid body system was proved by Kreml, Nečasová and Piasecki \cite{KrNePi20}, using relative energy estimates and a change of variables to compare solutions on different fluid domains. Further existence results address Navier-slip boundary conditions \cite{NeRaMy22} and general inflow–outflow conditions on the outer boundary \cite{ANR}.

The corresponding inviscid problem has also been studied. Local existence and uniqueness of sufficiently regular solutions to the incompressible Euler–rigid body system were established in \cite{Tucsnak-JFA,Wang-Zang}. These solutions provide the reference dynamics for the singular limit considered here.

The analysis of the low Mach number limit includes the foundational work of Klainerman and Majda \cite{KM1}. For viscous compressible flows, Desjardins and Grenier \cite{DesGre} studied the limit in the whole space, while Desjardins, Grenier, Lions and Masmoudi \cite{DGLM} treated bounded domains with Dirichlet boundary conditions. Broader accounts of incompressible limits and the methods used to study them can be found in \cite{Gallag,SCH2,MAS1,FN}.

Related results are available for exterior flows with prescribed body motion. Kračmar, Nečasová and Novotný \cite{KNN} proved the existence of weak solutions for a compressible viscous fluid around a rotating obstacle. Eiter, Nečasová and Oschmann \cite{ENO} established weak-strong uniqueness and justified the low Mach number limit for exterior flows around a body undergoing prescribed rigid motion. In the latter result, the viscosity remains positive and the limiting equations are the incompressible Navier--Stokes equations.

The simultaneous incompressible and inviscid limit has been investigated in several fluid models. Masmoudi \cite{MAS5} established convergence from the isentropic compressible Navier--Stokes system to the incompressible Euler equations as the Mach number and viscosity tend to zero. For 
heat-conducting fluids, Feireisl and Novotný \cite{Feireisl-Novotny-INVISCID} studied the full Navier--Stokes--Fourier system in a regime of small Mach number and large Reynolds and Péclet numbers, identifying the Euler--Boussinesq system as the limit. Inviscid incompressible limits on expanding domains were considered by Feireisl, Nečasová and Sun \cite{Eduard-Sarka-Sun}.

For the vanishing viscosity limit in the presence of a fixed no-slip boundary, Kato \cite{Kato-zero-viscosity} identified the decisive role of dissipation near the boundary. Assuming convergence of the initial data and the existence of a smooth incompressible Euler solution, convergence in the energy norm is equivalent to the vanishing of viscous dissipation in a boundary layer whose thickness is proportional to the viscosity. Sueur \cite{MR2993933} extended this criterion to an incompressible fluid coupled to a freely moving rigid body, obtaining convergence of both the fluid and body velocities. For compressible fluids in a fixed bounded domain, Sueur \cite{Sueur-compressible-inviscid} established a sufficient condition for the inviscid limit under no-slip boundary conditions, using the relative energy framework developed in \cite{FeNoSun,FeJiNo}. These results provide the boundary layer and relative energy methods on which the present analysis builds.

\vspace{.2cm}

The present work establishes a Kato-type sufficient criterion for the simultaneous low Mach number and vanishing viscosity limit of the three-dimensional compressible Navier-Stokes-rigid body system. The result extends the boundary layer approach for incompressible 
fluid-body interaction \cite{MR2993933} to a setting in which incompressibility must also emerge in the limit. The body's translation and rotation are determined by Newton's laws throughout the evolution, so the convergence concerns the full coupled dynamics. The central issue is to control density fluctuations and viscous boundary layer contributions together with the differences between the viscous and inviscid body motions.

More precisely, we consider families of finite-energy weak solutions whose initial relative energy with respect to a strong incompressible 
Euler-rigid body solution tends to zero. We prove convergence in relative energy as the Mach number and viscosity vanish simultaneously, on every time interval on which the limiting solution remains regular. The sufficient boundary layer condition combines vanishing viscous dissipation with a weighted smallness condition on the fluid momentum relative to the body, both measured in a neighbourhood of the body surface whose thickness is proportional to the viscosity. The convergence controls the fluid motion and the translational and angular velocities of the rigid body.

The proof is organized around a relative energy estimate for the complete fluid--body system. Working in coordinates attached to the body places the viscous and inviscid problems on a common fixed exterior domain, while retaining the translational and rotational velocities in the transformed equations. A boundary layer corrector makes the Euler comparison field compatible with the no-slip condition. The central analytical task is to estimate the resulting correction terms together with the contributions of compressible transport and body motion, so that the fluid and solid errors are controlled within the same inequality. For initial data with vanishing relative energy, this approach shows that suitable control of the fluid relative to the body in a shrinking neighbourhood of the interface suffices to recover the full incompressible Euler--rigid body dynamics.

The paper is organized as follows. Section \ref{sec2} describes the coupled fluid–rigid body system. In Section \ref{sec3}, we reformulate the viscous and inviscid problems in coordinates attached to the body and state the main convergence result. Section \ref{sec4} is devoted to the proof of the main result: we present the relative energy inequality, construct a Kato-type boundary layer corrector, and combine the resulting estimates to establish the theorem.
\section{Model description}\label{sec2}

Let  $\mathcal{S}(t)\subset
\R^3$ be a closed, bounded domain occupied by a rigid body emersed in some  viscous compressible fluid. Accordingly, the fluid domain is denoted by $\mathcal{F}(t)=\R^3\setminus \overline{\mathcal{S}(t)
}$ at a time moment $t\in \R^{+}.$ The initial domain of the rigid body is denoted by $\mc{S}_0$ and is assumed to have a smooth boundary. Correspondingly, $\mc{F}_0 =\mathbb{R}^{3} \setminus \mc{S}_0$ is  the initial fluid domain. The interface between the body and the fluid is denoted
by $\partial \mathcal{S}(t)$. For $T > 0$  we introduce the following notations for the 
space-time domains 
\begin{equation*}
\begin{array}{l}
{Q}_{\mathcal{F}}= \bigcup_{t\in (0,T)} \{t\} \times \mathcal{F}(t),
%\{(t,\x)\in \R^4:\,\,t\in (0,T),\quad \x\in \mathcal{F}%(t)\},
\quad 
{Q}_{\mathcal{S}}=\bigcup_{t\in (0,T)} \{t\} \times \mathcal{S}(t) , 
\quad 
{Q}_{\partial \mathcal{S}}=\bigcup_{t\in (0,T)} \{t\} \times \partial\mathcal{S}(t) .
%\{(t,\x)\in{\R}^4:\,\,t\in (0,T),\quad \x\in \partial\mathcal{S}(t)\},
%\{(t,\x)\in{\R}^4:\,\,t\in (0,T),\quad \x\in\mathcal{S}(t)\},
%\\
%Q_T= (0,T) \times \R^3.
\end{array}
\end{equation*}%

%For $T>0$ given and for any time $t \in (0,T)$, let $\mathcal{S}(t) \subset \mathbb{R}^3$ denote a closed, bounded and simply connected  rigid body. We assume that the rest of the space, i.e., 
%$\mathbb{R}^{3} \setminus \mathcal{S}(t)=\mathcal{F}(t)$ is filled with a viscous compressible fluid. 
%The initial domain of the rigid body is denoted by $\mc{S}_0$ and is assumed to have a smooth boundary. Correspondingly, $\mc{F}_0 =\mathbb{R}^{3} \setminus \mc{S}_0$ is  the initial fluid domain.

The system of  rigid body moving in a fluid is in the first instance a moving domain problem in an inertial frame in which the velocity of the rigid body is described by
\begin{equation} \label{def-U_S} 
U_{\mc{S}}(t,x)= h'(t)+  \Omega(t) \times (x-h(t)) \ \  \text{for} \ \ (t,x) \in Q_{\mc S}. 
\end{equation} 
Here, $h'(t)$ is the linear velocity of the centre of mass $h(t)$ and $ \Omega(t)$ is the angular velocity of the body. 
The solid domain at time $t$ in the same inertial frame is given by
$\mc{S}(t)= \left\{h(t)+ \mathbb Q(t)y \mid y\in \mc{S}_{0}\right\}, $ where $\Q(t) \in SO(3)$ is associated to the rotation of the rigid body such that $\Q'  = \Omega \times \Q , \ \ \Q(0) = \mathbb{I}$.

%{Mathematically, it will turn out useful to define $U_{\mc{S}}$ for all $(t,x)\in  (0,T)\times \mathbb R^3$.}

Let $\varrho$, $U$ and $P(\varrho)$ be the density, velocity and pressure of the scaled viscous compressible  fluid, respectively, which satisfy 
\begin{align} 
\frac{\partial \varrho}{\partial t} + \operatorname{div}\,(\varrho U) &=0, \ \ 
\varrho(x) \geq 0  \quad \mbox{ in } Q_{\mc F} ,  \label{mass:fluid} \\ 
\frac{\partial }{\partial t}(\varrho U) + \operatorname{div}\left(\varrho U
\otimes U\right) + \frac{1}{\varepsilon^2}\nabla P(\varrho) &= \nu\operatorname{div}(\Sigma(U)) \quad  \ \mbox{ in } Q_{\mathcal{F}}.
\label{momentum:fluid}
\end{align}
Here $\varepsilon$, $\nu$ are respectively the Mach number and inverse of Reynolds number, and the viscous stress tensor is defined by  
$$\Sigma (U) = \mu\left(\nabla U + \nabla^\top U - \frac{2}{3}\operatorname{div}U\mathbb{I}\right) + \eta\operatorname{div}U\mathbb{I},$$
where $\mu>0$,  $\eta \geq 0$ are two viscosity coefficients. 
We assume that the flow is in the barotropic regime and we focus on the isentropic case where the relation between $P(\varrho)$ and $\rho$ is given by the constitutive law:
\begin{equation}\label{P-form}
P(\varrho)=a\varrho^{\gamma},\mathrm{\, with \,}\ a>0 \ \mathrm{ and \ the\ adiabatic\ constant \ }\gamma>3/2.
\end{equation}
Let $\varrho_{\mc{S}}$ denote the density of the rigid body. Then $\displaystyle m = \int_{\mc{S}(t)} \varrho_{\mc{S}}(t,x) \, dx$ is its total mass, which is constant in time and does not change under the coordinate transformation below. Further, 
the moment of inertia $J(t)$ is defined by
$$J(t) =  \int\limits_{\mc{S}(t)} \rho_{\mc{S}}(t,x) 
 \Big(|x-h(t)|^2\mathbb{I} - (x-h(t))\otimes (x-h(t))\Big)\, dx. $$
The ordinary differential equations modeling the dynamics of the rigid body then read as 
\begin{align} 
mh''&= -\int\limits_{\partial \mc{S}(t)} \left(\nu\Sigma (U)- \frac{1}{\veps^2} P(\varrho)\mathbb{I} \right) N\, d\Gamma ,\label{body1}
\\ 
(J{\Omega})' &= -\int\limits_{\partial \mc{S}(t)} (x-h) \times \left(\nu\Sigma (U)- \frac{1}{\veps^2}P(\varrho)\mathbb{I} \right) N\, d\Gamma,\label{body-2}
\end{align} 
where $N$ is the unit outward normal to the boundary of $\mathcal{F}(t)$, i.e., directed towards $\mathcal{S}(t)$. We consider $J(0)=J_0$
and it satisfies the following relation $J=\Q J_0 \Q^T.$ 
We further assume that the density and the velocity of the fluid satisfy:
\begin{equation}\label{13:10}
\varrho(t,x)\rightarrow 1,\quad U(t,x) \rightarrow 0 \,  \mbox{ as } \,  |x| \rightarrow \infty .
\end{equation}
%where $\overline{\rho}>0$.
The system is supplemented with  the following initial conditions: \begin{equation}\label{initialcond} 
\begin{aligned}
&\varrho(0,x)=\rho_0 \geq 0 ,  \ \,   \varrho U(0,x)=q_0 \,  \text{ in } \mc{F}_0,  \\
&h(0)=0, \ \  h'(0)=\ell_{0}, \ \ {\Omega}(0)= \omega_{0}. 
\end{aligned}
\end{equation}
Finally, 
we consider the continuity condition on the  fluid--rigid  interface, namely  
\begin{equation}\label{b.c-1} 
U= U_{\mc{S}} \quad \mbox{ on } Q_{\partial \mc{S}}.
\end{equation}

We want to investigate the issue of the incompressible inviscid limit of the system \eqref{mass:fluid}--\eqref{b.c-1}, that is, when $\veps \to 0$ and $\nu \to 0$. For the limiting system, let us denote by $\mathcal{S}^E(t)\subset
\R^3$, the domain occupied by the rigid body  and   $\mathcal{F}^E(t)=\R^3\setminus \overline{\mathcal{S}^E(t)
}$, the domain filled with the fluid at a time moment $t\in \R^{+}.$ The interface in this case is denoted by $\partial \mc{S}^E(t)$.   We assume that the initial
position  of the rigid body, denoted by $\mc{S}^E(0)$,   is the same as the  former system, i.e.,  $\mathcal{S}^E(0)=\mc{S}_0$. 
Accordingly, the initial fluid domain is  
$\mathcal{F}^E(0):=\mc{F}_0$.
For $T> 0$, we now introduce the following notations for the space-time domains: 
\begin{equation*}
\begin{array}{l}
{Q}_{\mathcal{F}}^E= \bigcup_{t\in (0,T)} \{t\} \times \mathcal{F}^E(t),
\quad 
{Q}_{\mathcal{S}}^E=\bigcup_{t\in (0,T)} \{t\} \times \mathcal{S}^E(t),
\quad 
{Q}_{\partial \mathcal{S}}^E=\bigcup_{t\in (0,T)} \{t\} \times \partial\mathcal{S}^E(t) . 
\end{array}
\end{equation*} 
When $\varepsilon\rightarrow 0$ and $\nu \rightarrow 0$, we anticipate that formally the system \eqref{mass:fluid}--\eqref{b.c-1} degenerates into the following target system:
\begin{align}   
\frac{\partial}{\partial t} U^E  + \operatorname{div}\left(U^{E}
\otimes U^{E}\right) + \nabla P^{E} &= 0 \quad \mbox{ in } Q^E_{\mc F}, 
\label{momentum:fluidE} \\
\operatorname{div}\,U^{E} &=0  \quad \mbox{ in } Q^E_{\mc F} , 
\label{mass:fluidE}
\end{align}
with the interface condition: 
\begin{align} 
U^{E}\cdot N^E  = U^{E}_{\mc{S}}\cdot N^E  \quad \mbox{ on } Q_{\partial \mc{S}}^{E} , 
\label{b.c-1E}
\end{align}
the dynamics of the rigid body:
\begin{align} m(h^{E})''&= \int\limits_{\partial \mc{S}^{E}(t)} P^{E} N^E\, d\Gamma ,\label{body1E}
\\ 
(J^{E}{\Omega}^{E})' &= \int\limits_{\partial \mc{S}^{E}(t)} (x-h^{E}) \times P^{E} N^E \, d\Gamma  ,\label{body-2E}
\end{align} 
and the initial conditions: 
\begin{equation}\label{initialcondE} 
 U^{E}(0,\cdot)=U_0^{E}  , \ \,  h^{E}(0)=0, \ \,  (h^{E})'(0)=\ell_{0}^{E}, \ \, {\Omega}^{E}(0)=\omega^{E}_{0}. 
\end{equation}
In the above set of equations, the velocity of the rigid body is given by 
\begin{equation*}
U^{E}_{\mc{S}}(t,x)= (h^{E})'(t) + {\Omega}^{E}(t) \times (x-h^{E}(t)) , 
\end{equation*}
where $(h^E)'(t)$ is the linear velocity of the centre of mass $h^E(t)$ and $\Omega^E(t)$ is the angular velocity of the body. 
The solid domain at time $t$  then reads as 
$\mc{S}^{E}(t)= \left\{h^{E}(t)+ \Q^{E}(t)y \mid y\in \mc{S}_{0}\right\}$,
where the rotational matrix $\Q^{E}$ satisfies 
$(\Q^{E})'= {\Omega}^{E} \times \Q^{E}, \quad \Q^{E}(0)=\mathbb{I}$.
Finally,  $J^E(0)=J_0$ and  $J^{E}$ is given by 
$J^{E} = {\Q}^{E} J_0 (\Q^{E})^{\top}$.

%%%%%%%%%%%%%%%%%%%%%%%%%%%%%%%

\section{Reformulation and the main result}\label{sec3}

 In this section, we  rewrite the systems \eqref{mass:fluid}--\eqref{b.c-1} and \eqref{momentum:fluidE}--\eqref{initialcondE} in a fixed domain and  analyse the existence results of such transformed systems.

 \subsection{Case of the Navier--Stokes equations} 
 
 We apply a global change of variables that transforms the system in such a way that it is formulated in a frame which is attached to the rigid body (see, for instance, \cite{Galdi-book-chapter,Book-Sarka-et-al}).  At time $t=0$, the two frames are assumed to coincide with $h(0) =0$. Hence the transformed system is posed on the fixed domain  $(0,T)\times\mc{F}_{0}$ via the following change of variables: 
\begin{equation*}
u(t,y) =\Q(t)^{\top}U(t, \Q(t)y+h(t)), \quad \rho(y,t)=\varrho(t, \Q(t)y+h(t)).
\end{equation*}
%\begin{equation*}
% p(y,t)=P(Q(t)y+h(t),t).
%\end{equation*} 
%Further, the moment of inertia transforms to $$ J_0 = \Q(t)^\top J(t) \Q(t), \quad t\in (0,T).$$
By \eqref{def-U_S}
and the extension of $U_{\mc{S}}$ to the whole space, we have that 
\begin{equation*} 
u_{\mc{S}}(t,y) = \Q(t)^\top U_{\mc{S}}(t,\Q(t)y + h(t))  =  \ell(t) + \omega(t) \times y, \quad  (t,y)\in  (0,T) \times \mathbb R^3,
\end{equation*} % \F_0
where $\ell(t) = \Q(t)^{\top}h'(t)$ is the transformed linear velocity and $\omega(t)=\Q(t)^{\top}{\Omega}(t)$ the transformed angular velocity of the rigid body. 
Equations \eqref{mass:fluid}--\eqref{initialcond} can be rewritten in the fixed domain as 
\begin{align} 
\frac{\partial\rho}{\partial t} + \operatorname{div}\,(\rho (u-u_{\mc{S}})) &=0, \ \,  \rho(y) \geq 0 \quad \mbox{ in } (0,T)\times  \mc{F}_0 , \label{chg of var mass:fluid} \\ \frac{\partial }{\partial t}(\rho u) + \operatorname{div}\left[\rho(u-u_{\mc{S}})\otimes u\right] + \rho \omega \times u +\frac{1}{\varepsilon^2} \nabla p &= \nu\operatorname{div}(\sigma(u)) \quad \ \mbox{ in } (0,T)\times  \mc{F}_0 , \label{chg of var momentum:fluid} 
\end{align} 
\begin{align} 
m\ell'&= -\int\limits_{\partial \mc{S}_{0}} \left(\nu\sigma (u)-
\frac{1}{\varepsilon^2} p(\rho)\mathbb{I}\right) n\, d\Gamma + { m \ell \times \omega}, \label{chg of var linear momentum:body}\\ 
J_{0}\omega' &= -\int\limits_{\partial \mc{S}_{0}} y \times \left(\nu\sigma (u)- \frac{1}{\varepsilon^2}p(\rho)\mathbb{I}\right) n\, d\Gamma  + J_{0}\omega \times \omega, \label{chg of var angular momentum:body} 
\end{align} 
with the initial conditions
\begin{equation}
\begin{aligned}\label{chg of var initial cond} 
&\rho(0,y)=\rho_0 \geq 0,\  \, \rho u(0,y)=q_0 \ \, \text{in }\, \mc F_0, 
\\
&h(0)=0,\ \, \ell(0)=\ell_{0},\ \, \omega(0)=\omega_{0}, 
\end{aligned}
\end{equation} 
where $n(t,y) = \Q(t)^\top N(\Q(t)y + h(t))$ for $(t,y) \in (0,T)\times \partial \mc{S}_0$ denoting the inward pointing normal to $\partial \mc{S}_0$ and

$$ \sigma (u) = \mu\left(\nabla u + \nabla^\top u - \frac{2}{3}\operatorname{div}u\mathbb{I}\right) + \eta\operatorname{div}u\mathbb{I}.$$
It follows from relation \eqref{P-form} that $p(\rho)$ is given by:
\begin{equation}\label{P-form-new}
p(\rho)=a\rho^{\gamma} ,\mathrm{\, with \,}\ a>0 \ \mathrm{ and \ the\ adiabatic\ constant \ }\gamma>3/2.
\end{equation}
 It follows from \eqref{13:10} that $(\rho, u)$ satisfies
$\rho(t, y)\rightarrow 1,\ u(t,y) \rightarrow 0$ as $|y| \rightarrow \infty$. We also introduce here the notation $H(\rho)$ for the pressure potential:
\begin{equation}\label{P-potential}
 H(\rho)=\rho\int\limits_{1}^{\rho} \frac{p(z)}{z^2}\ dz.
 \end{equation}
In terms of the new state variable, the interface condition \eqref{b.c-1} is  
\begin{equation}\label{chg of variable boundary:no-slip} 
u= u_{\mc{S}} \quad \mbox{ on } Q_{\partial \mc{S}_0}.
\end{equation}
Now, we give the 
meaning of solution to the  
system 
\eqref{chg of var mass:fluid}--\eqref{chg of variable boundary:no-slip}. In the rest of the paper, $\rho$, $u$ stand for the following extended version: $\forall\  t\in [0,T]$, 
  \begin{equation*}
  \rho (t,y) = \begin{cases}
  \rho (t,y), \quad y \in \mc{F}_0, \\
  \rho_{\mc{S}}(t,y), \quad y \in \mc{S}_0 , 
  \end{cases} \quad 
  u(t,y) = \begin{cases}
  u(t,y), \quad  y \in\mc{F}_0, \\
  u_{\mc{S}}(t,y),\quad y \in\mc{S}_0.
  \end{cases}
  \end{equation*}
Let us introduce the following space
\begin{equation*}
\mc{H}=\left\{ \phi \in L^2(\mathbb{R}^3) \mid  \mathbb{D}(\phi)=0 \mbox{ in some open neighborhood of 
}\overline{\mc{S}_{0}} \right\} ,
\end{equation*}
where $\mathbb{D}(\phi)= \frac{1}{2}(\nabla \phi + \nabla^\top \phi)$. 
According to \cite[Lemma 1.1, p.18]{Temam-plasticity},  for any $\phi \in \mc{H}$, there exist $\ell_{\phi} \in \mathbb{R}^3$ and $\omega_{\phi} \in \mathbb{R}^3$ such that 
$\phi(y)= \ell_{\phi}
+ \omega_{\phi} \times y  \quad \mbox{ for all }y \in \mc{S}_{0}$. 
We hereby denote $\phi_{\mc S}(y) : = \ell_\phi + \omega_\phi \times y$  for all $y \in \mc{S}_{0}.$ We consider the space $L^2(\mathbb{R}^3)$ with the following inner product: 
\begin{equation*} 
\left(\phi, \psi\right)_{\mc{H}} = \int\limits_{\mc{F}_{0}} \rho\phi \cdot \psi\,dy+ \int\limits_{\mc{S}_0} \rho_{\mc{S}} \phi \cdot \psi \, dy. 
\end{equation*}
In fact, for $\phi, \psi \in \mc{H}$, we have 
\begin{equation}\label{inner product:W} 
\left(\phi, \psi\right)_{\mc{H}} = \int\limits_{\mc{F}_{0}} \rho \phi \cdot \psi\,dy + m\ell_{\phi}\cdot \ell_{\psi} + J_{0} \omega_{\phi}\cdot \omega_{\psi}. 
\end{equation}
 The norm associated with the above inner product is denoted by $\|\cdot\|_{\mc{H}}=\left(\cdot, \cdot\right)^{1/2}_{\mc{H}}$. We  further define the space 
\begin{equation*} 
 {\mc{V}}=\left\{ \phi \in {\mc H} \mid \int\limits_{\mathbb{R}^{3}} |\nabla \phi(y)|^2  (1+|y|^2)\, dy < \infty \right\}, 
 \end{equation*}
 with its norm defined as
 \begin{equation}\label{norm-V}
 \|\phi\|_{\mc{V}} = \|\phi\|_{\mc H} + \|\nabla \phi\|_{L^2(\mathbb{R}^{3}, (1+|y|^2)^{\frac{1}{2}} dy )}. 
 \end{equation}
 
Let us  write  the following definition.  
\begin{defin}\label{23:54}
Let $T> 0$ and  
assume that the initial data satisfy 
  \begin{align} \label{init}
  \rho_{0} \in L^{\gamma}(\mathbb{R}^3),\quad \rho_{{0}} \geq 0 &\mbox{ a.e. in }\mathbb{R}^3,\\
  \label{init1}
  q_{{0}} \in L^{\frac{2\gamma}{\gamma+1}}(\mathbb{R}^3), \quad q_{{0}}\mathds{1}_{\{\rho_{0}=0\}}=0 &\mbox{ a.e. in }\mathbb{R}^3,\quad \dfrac{|q_{{0}}|^2}{\rho_{0}}\mathds{1}_{\{\rho_{0}>0\}}\in L^1(\mathbb{R}^3),\\
  \label{init2}
  \frac{q_0}{\rho_0}= \ell_0+ \omega_0 \times y &\mbox{ on } \mc{S}_0 , \ \ \ell_0,\, \omega_0\in \R^3.
  \end{align}
%$\rho_0 \in L^\gamma(\mathbb R^3)$ and $q_0\in L^\frac{2\gamma}{\gamma+1}(\mathbb R^3)$, $\gamma>3/2$. 
Then, a pair $(\rho,u)$ is a weak solution to system \eqref{chg of var mass:fluid}--\eqref{chg of variable boundary:no-slip} if the following conditions hold true:
\begin{itemize} \item $\rho \geq 0$,  
${\rho-1 \in L^{\infty}(0,T; L^\gamma(\mathbb R^3))}$.

\item $u \in L^\infty(0,T; \mc H) \cap L^2(0,T; H^1(\mathbb R^3) )$. 

\item The equation of continuity \eqref{chg of var mass:fluid} is satisfied in the weak sense, i.e., 
\begin{equation}\label{weakform-continuity} \int\limits_{0}^{t}\int\limits_{\mathbb R^3} \left[\rho \frac{\partial \psi}{\partial s} + \rho (u-u_{\mc{S}})\cdot  \nabla \psi\right]\,dy\,ds = \int\limits_{\mathbb R^3} \rho(t,y)\psi(t,y) \, dy  - 
\int\limits_{\mathbb R^3}  \rho_0\psi(0,y) \, dy ,  
\end{equation} 
for any test function $\psi \in C_c^{\infty}([0,T]\times \mathbb R^3)$.

Also, a renormalized continuity equation holds in a weak sense, i.e., \begin{equation*}  \int\limits_{0}^{t}\int\limits_{\mathbb{R}^3} b(\rho)\left[\frac{\partial \psi}{\partial s} + (u-u_{\mc{S}}) \cdot  \nabla \psi\right]\,dy\,ds 
= \int\limits_{\mathbb R^3}b(\rho(t,y))\phi(t,y) \, dy  - \int\limits_{\mathbb{R}^3} b(\rho_0)\psi(0,y) \, dy ,
\end{equation*} 
for any test function $\psi \in C_c^{\infty}([0,T]\times \mathbb{R}^3)$ and for all $b \in C^1(\mathbb{R})$. 
\item Balance of linear momentum holds in a weak sense, i.e., for all %\phi \in \C^H^1(0,T;L^2(\mathbb{R}^3)) \cap L^2(0,T;H^1(\mathbb{R}^3))$ and for all $t \in [0,T]$, 
$\phi \in C^\infty([0,T]; \mc{H} \cap  C^\infty_c(\mathbb R^3))$,  the following relation holds:
\begin{multline}\label{weakform-momentum}
 \int\limits_{\mc{F}_{0}} \rho u(t,y)\cdot \phi(t,y)\,dy + m\ell(t)\cdot \ell_{\phi}(t)
 + J_{0}\omega(t)\cdot \omega_{\phi}(t)
  - \int\limits_{\mc{F}_{0}} q_0 \cdot \phi(0,y)\,dy - m\ell_0\cdot \ell_{\phi}(0) - J_{0}\omega_0\cdot \omega_{\phi}(0)  \\
=  \int\limits_{0}^{t} \int\limits_{\mc{F}_{0}}\rho u \cdot \frac{\partial \phi}{\partial s}\,dy\,ds + \int\limits_{0}^{t}m\ell\cdot \ell'_{\phi}\,ds
+ \int\limits_{0}^{t}J_{0}\omega\cdot \omega'_{\phi}\,ds +\int\limits_{0}^{t}\int\limits_{\mc{F}_{0}} \left[ \rho(u-u_{\mc{S}})\otimes u \right] : \mathbb D
(\phi)\,dy\,ds   
  - \int\limits_{0}^{t}\int\limits_{\mc{F}_{0}} \operatorname{det} (\rho \omega, u, \phi)\,dy\,ds \\ + \int\limits_{0}^{t}
\operatorname{det}(m \ell,\omega,\ell_{\phi})\,ds
 + \int\limits_{0}^{t} \operatorname{det}(J_{0}\omega,\omega,\omega_{\phi}) \,ds  
 -  \int\limits_{0}^{t}\int\limits_{\mc{F}_{0}} \left(\nu \sigma(u) - \frac{1}{\varepsilon^2}p(\rho)\mathbb{I}\right): \mathbb D (\phi)\,dy\,ds .
 \end{multline}

\item Moreover, the following inequality holds: 
\begin{align}\label{energy}
\frac{1}{2}\int\limits_{\R^3}\rho |u|^2 (t,\cdot) \, dy  + \frac{1}{\veps^2} \int\limits_{\F_0} H(\rho)(t,\cdot) \, dy + \nu \int\limits_0^t \int\limits_{\F_0} \sigma(u) : \mathbb{D}(u) \, dy ds \notag \\
\leq \frac{1}{2}\int\limits_{\R^3} \frac{|q_0|^2}{\rho_0} \mathds{1}_{\{\rho_0>0\}}dy + \frac{1}{\veps^2}\int\limits_{\mc F_0} H(\rho_0) \, dy 
\end{align}
for a.a. $t\in (0,T)$. 
\end{itemize}
\end{defin}

\begin{remark}\label{Remark-detailed-norms}
Observe that, we can use the definition of the inner product  of the space $\mc{H}$ to write
\begin{multline*}
\int\limits_{\mc{F}_{0}} \rho u(t,y)\cdot \phi(t,y)\,dy + m\ell(t)\cdot \ell_{\phi}(t)
 + J_{0}\omega(t)\cdot \omega_{\phi}(t)
  - \int\limits_{\mc{F}_{0}} q_0\cdot \phi(0,y)\,dy - m\ell_0\cdot \ell_{\phi}(0) - J_{0}\omega_0\cdot \omega_{\phi}(0)\\ = \left(u,\phi\right)_{\mc{H}}(t) - \left(\frac{q_0}{\rho_0},\phi(\cdot,0)\right)_{\mc{H}}
\end{multline*}
in the momentum balance \eqref{weakform-momentum}.
Similarly, 
\begin{equation*}
\int\limits_{\mc{F}_{0}}\rho u \cdot \frac{\partial \phi}{\partial s}\,dy\,ds + m\ell\cdot \ell'_{\phi}\,ds
+ J_{0}\omega\cdot \omega'_{\phi}\,ds = \left(u,\frac{\partial \phi}{\partial s}\right)_{\mc{H}}.
\end{equation*}
If there is no confusion, for simplicity, we shall often abuse the notations:  
%$(\phi, \psi)_{\mc{H}} = \int_{\mathbb R^3} \rho \phi\cdot  \psi \, dy$
%and according to the definition, one can observe that
%%%
\begin{align}\label{convention}
\int\limits_{\mathbb R^3} \rho \phi \cdot \psi \, dy   =  (\phi, \psi)_{\mc H} =  \int\limits_{\F_0} \rho \phi\cdot  \psi \, dy  + m\ell_\phi\cdot \ell_\psi + J_0 \omega_\phi \cdot \omega_\psi .
\end{align}
Therefore, $u\in L^\infty(0,T;\mc H)$ is equivalent to write $\rho|u|^2 \in L^{\infty}(0,T; L^1(\mathbb{R}^3))$ with 
\begin{align}
\int\limits_{\R^3}\rho |u|^2(t,y)\, dy  = \int\limits_{\F_0} \rho|u|^2(t,y) \, dy + m|\ell(t)|^2 + J_0|\omega(t)|^2 = \| u(t)\|^2_{\mc H} . 
\end{align}
\end{remark}

{ We can establish the following existence result for the weak solution to system \eqref{chg of var mass:fluid}--\eqref{chg of variable boundary:no-slip}, see for instance \cite[Theorem 4.1]{Feireisl-compressible-Rigid} and \cite{KNN}.}
 \begin{theorem}\label{NS-body}
  Let $\mc{S}_0 \subset \mathbb{R}^3$ be a $C^{1,1}$ bounded,  closed, simply connected set and $\mc{F}_0= \mathbb{R}^3\setminus \mc{S}_0$. 
 % \begin{align} \label{init}
 % \rho_{0} \in L^{\gamma}(\mathbb{R}^3),\quad \rho_{{0}} \geq 0 &\mbox{ a.e. in }\mathbb{R}^3,\\
 % \label{init1}
 % q_{{0}} \in L^{\frac{2\gamma}{\gamma+1}}(\mathbb{R}^3), \quad q_{{0}}\mathds{1}_{\{\rho_{0}=0\}}=0 &\mbox{ a.e. in }\mathbb{R}^3,\quad \dfrac{|q_{{0}}|^2}{\rho_{0}}\mathds{1}_{\{\rho_{0}>0\}}\in L^1(\mathbb{R}^3),\\
%  \label{init2}
%  \frac{q_0}{\rho_0}= \ell_0+ \omega_0 \times (y-h_0)&\mbox{ on } \mc{S}_0 \mbox{ and }\ell_0,\, \omega_0, \, h_0 \in \mathbb{R}^3.
%  \end{align}
 Then,  for any $T>0$, with the choices of initial data \eqref{init}--\eqref{init2}, there exists a weak solution $(\rho,u)$ to system \eqref{chg of var mass:fluid}--\eqref{chg of variable boundary:no-slip}   in the sense of Definition \ref{23:54}. 
 %
 %Moreover, the following energy inequality holds for almost every $t\in (0,T)$:
%\begin{align}\label{energy}
%\int\limits_{\mathbb{R}^3}\frac{1}{2} \rho|u|^2(t, \cdot) \, dy  + \frac{1}{\varepsilon^2} \int\limits_{\mc{F}_0}  H(\rho)(t,\cdot) d y & 
%+\nu \int\limits_0^t\int\limits_{\mc{F}_0} \sigma(u): \mathbb D(u) \, dy \, ds  \notag \\
% & \leq  
 % \int\limits_{\mathbb{R}^3} \frac{1}{2}\dfrac{|q_{{0}}|^2}{\rho_{0}}\mathds{1}_{\{\rho_{0}>0\}} \, dy +\frac{1}{\varepsilon^2} \int\limits_{\mc{F}_0} H(\rho_0) \, dy .
%\end{align}
\end{theorem}

 %{\color{blue}  Let us also mention the existence theory for the exterior problem, see \cite{KNN}.where only rotation of the body was considered.  \begin{theorem}\label{thm:existence}  Let $\gamma>\frac{3}{2}$. 
    % Let $\rho_\infty > 0$, and $\bomega,\,\vc a_\infty\in\R^3$ with $\bomega \times \vc a_\infty = 0$. Assume further that $(\rho_0, \bu_0)\in L^1_{\mathrm{loc}}(\Omega)\times L^1_{\mathrm{loc}}(\Omega;\R^3)$ with
    % \begin{align}
    %     \rho_0 \geq 0, \quad P(\rho_0 | \rho_\infty) \in L^1(\Omega), \quad \rho_0 |\bu_0 - \vc a_\infty|^2 \in L^1(\Omega)
    % \end{align}
    % and $\vc f \in L^\infty(0,T; L^1(\Omega) \cap L^\infty(\Omega) )$. Then there exists a finite-energy weak solution to system \eqref{chg of var mass:fluid}-- \eqref{P-form-new} in the sense of Definition~\ref{23:54}.\end{theorem}}

%%%%%%%%%%%%%%%%%%%%%%%%%%%

\subsection{Case of the Euler equation}    Let us consider ``incompressible Euler equation + rigid body" system  \eqref{momentum:fluidE}--\eqref{initialcondE}. We consider the following change of variables to rewrite the system in a fixed domain:
\begin{equation*} u^{E}(t,y) =(\Q(t)^{E})^{\top}U^{E}(t, \Q^{E}(t)y+h^{E}(t)), \quad  p^{E}(t,y)=P^{E}(t, \Q^{E}(t)y+h^{E}(t)),
\end{equation*}
\begin{equation*}
\ell^{E}(t) = (\Q^{E}(t))^{\top}(h^{E})'(t),\quad \omega^{E}(t)=(\Q^{E}(t))^{\top}{\Omega^{E}}(t).
\end{equation*}
Under this change of variables, equations \eqref{momentum:fluidE}--\eqref{initialcondE} have been transformed in the fixed domain as 
%\operatorname{div}\left[u^{E}\otimes (u^{E}-u^{E}_{\mc{S}})\right]
\begin{align}  
\frac{\partial }{\partial t}u^{E} + (u^E - u^E_{\mc{S}}) \cdot \nabla u^E + \omega^{E} \times u^{E} + \nabla p^{E} &= 0 \ \ \qquad \mbox{ in } (0,T)\times \mc{F}_0 , \label{chg of var momentum:fluidE} \\
\operatorname{div}\,u^{E} &=0 \ \ \qquad \mbox{ in } (0,T)\times \mc{F}_0 , \label{chg of var mass:fluidE} \\
 u^{E}\cdot n &= u^{E}_{\mc{S}}\cdot n  \  \ \mbox{ on }  \ (0,T)\times \partial \mc{S}_{0} ,  \label{chg of var boundaryE} \end{align} 
\begin{align} 
m(\ell^{E})'&= \int\limits_{\partial \mc{S}_{0}} p^{E} n\, d\Gamma + m\ell^{E} \times \omega^{E}, \label{chg of var linear momentum:bodyE}\\ 
J_{0}(\omega^{E})' &= \int\limits_{\partial \mc{S}_{0}} y \times p^{E} n\, d\Gamma  + J_{0}\omega^{E} \times \omega^{E}, \label{chg of var angular momentum:bodyE} \end{align} 
\begin{equation}\label{chg of var initial condE}  u^{E}(0,y)=u_0^{E} ,\  h^{E}(0)=0,\ \ell^{E}(0)=\ell^{E}_{0},\ \omega^{E}(0)=\omega^{E}_{0}, \end{equation}
with 
\begin{equation*}
u^{E}_{\mc{S}}(t,y)= \ell^{E}(t) + \omega^{E}(t)\times y \ \text{ for } \ y \in \mc S_0.
\end{equation*}

%{\color{magenta} We need to introduce some functional spaces to state the existence and uniqueness of classical solutions to system \eqref{chg of var momentum:fluidE}--\eqref{chg of var initial condE}:

 Let us  define the space 
\begin{align}\label{space-H}
\mc{H}^E := \left\{ \phi \in L^2(\mathbb R^3) \mid \operatorname{div} \phi = 0 \text{ in } \mathbb R^3, \ \ \mathbb D(\phi)=0 \text{ in } \mc S_0  \right\} ,
\end{align} 
Furthermore, for $\phi, \psi \in \mc{H}^E$, the inner product is defined by 
\begin{equation}\label{inner product:Euler} 
{\left(\phi, \psi\right)_{\mc{H}^E} = \int\limits_{\mc{F}_{0}}  \phi \cdot \psi\,dy + m\ell_{\phi}\cdot \ell_{\psi} + J_{0} \omega_{\phi}\cdot \omega_{\psi}. } 
\end{equation}
 The norm associated with the above inner product is denoted by $\|\cdot\|_{\mc{H}^E}=\left(\cdot, \cdot\right)^{1/2}_{\mc{H}^E}$.

We can follow \cite[Theorem 2.4]{Wang-Zang} and \cite[Theorem 5]{MR2993933} to obtain 
the following existence--uniqueness result of the solution to the ``incompressible Euler + rigid body" system.  
\begin{theorem}\label{Thm-Euler-body}
Let $s>5/2$, $u^E_0\in H^s(\mc F_0)$, $\mathrm{curl}\,u^E_0\mid_{\mc F_0}$ is compactly supported and $u^E_0\cdot n = (\ell^E_0 + \omega^E_0 \times y)\cdot n$ on $\partial {\mc S}_0$. Then, there exists some $T>0$ and a unique solution $(u^E, p^E, \ell^E, \omega^E)$ to \eqref{chg of var momentum:fluidE}--\eqref{chg of var initial condE}   such that
\begin{equation}
\label{regularity-Euler}
\begin{aligned}
u^E \in C([0,T]; H^s(\mc F_0)), \ \ &\nabla p^E \in C([0,T]; H^{s-1}(\mc F_0)) ,\ \nabla u^E \in C([0,T];L^2(\mc F_0, (1+|y|^2)^{\frac12} dy)), \\
& \qquad \qquad \qquad  \ell^E , \, \omega^E \in C^1([0,T]; \mathbb R^3) . 
\end{aligned}
\end{equation}
\end{theorem}

%\begin{remark}\label{Remark-Euler-Body} From  \eqref{regularity-Euler}, one may deduce that 
%\begin{align*}
%p^E\in C([0,T]; H^{s}(\mc F_0)), \  \,   u^E_t \in C([0,T]; H^{s-1}(\mc F_0)), \ \, \ p^E_t \in C([0,T]; H^{s-1}(\mc F_0)) .
%\end{align*} \begin{align*}u^E , \, \nabla u^E, \, \nabla p^E \in L^\infty((0,T)\times \mc F_0), \, \text{ provided } \, s>4 . \end{align*}Moreover, for any $t\in [0,T]$, one has  \begin{equation*}\frac{1}{2}\|u^{E}(t)\|^2_{\mc{H}^E}=\frac{1}{2}\|u_0^{E}\|^2_{\mc{H}^E}. \end{equation*}\end{remark}

%{\color{red}The existence-uniqueness result is already available in \cite[Theorem 5]{MR2993933}:} 
%\begin{theorem}\label{euler-body}
% Let $\lambda \in (0,1)$ be given and $u_0^{E}\in \mc{H}^E$ such that $u_0^{E}|_{\mc{F}_0}\in H^1 \cap C^{1,\lambda}$ and $\operatorname{curl}u_0^{E}|_{\mc{F}_0}$ is compactly supported. Then there exists $T>0$ and a unique solution $u^{E}$ of \eqref{chg of var momentum:fluidE}--\eqref{chg of var initial condE} in $C^1([0,T];\mc{H}^E)$ such that $$\nabla u^{E}|_{(0,T)\times \mc{F}_0} \in C([0,T];L^2(\mc{F}_0,(1+|x|^2)\ dx)) \cap C_{\textnormal{weak}^{*}}([0,T];C^{0,\lambda}(\mc{F}_0)).$$
%Moreover, for any $t\in [0,T]$,
% \begin{equation*}
% \frac{1}{2}\|u^{E}(t)\|^2_{\mc{H}^E}=\frac{1}{2}\|u_0^{E}\|^2_{\mc{H}^E}.
 %\end{equation*}
%\end{theorem}

\subsection{Main result}

We can now state the  main result of the paper. Let us introduce the strips
\begin{equation}\label{Gamma-strip}
\Gamma_{\nu} := \left\{y\in \mc{F}_0\mid \d(y)< \nu\right\}\mbox{ with }\d(y)= \operatorname{dist}(y,\partial\mc{S}_0).
\end{equation}
\begin{theorem}\label{main-thm}
Let $\mc{S}_0 \subset \mathbb{R}^3$ be a smooth, bounded, closed, simply connected set and $\mc{F}_0= \mathbb{R}^3\setminus \mc{S}_0$.
Let $c>0$ be given and assume that  $(u^{E},p^{E},\ell^{E},\omega^{E})$ is the strong solution to system \eqref{chg of var momentum:fluidE}--\eqref{chg of var initial condE} corresponding to an
initial data $(u^{E}_0,\ell_0^{E},\omega_0^{E})$ as in \cref{Thm-Euler-body}. Let $(\rho_0,q_0)$ be an initial data satisfying 
\eqref{init}--\eqref{init2}, the pressure $p(\rho)$ and pressure potential $H(\rho)$ are given by \eqref{P-form-new} and \eqref{P-potential} respectively and $(\rho,u)$ be an associated weak solution to system \eqref{chg of var mass:fluid}--\eqref{chg of variable boundary:no-slip}  as given by \cref{NS-body}.
We further assume that
\begin{equation}\label{Thm-Assum-1}
\frac{1}{2}\int\limits_{\mathbb{R}^3}\rho_0\left|\frac{q_0}{\rho_0}-u_0^{E}\right|^2 dy  + \frac{1}{\varepsilon^2}\int\limits_{\mc{F}_0}\left(H(\rho_0)-\frac{\partial H}{\partial\rho}(1)(\rho_0-1)-H(1)\right) dy \rightarrow 0 \ \mbox{ as }\varepsilon\rightarrow 0,\ \nu\rightarrow 0,
\end{equation}
and 
\begin{equation}\label{Thm-assum-2}
\nu \left\| \frac{\rho(u-u_{\mc S})}{d^2}\right\|^2_{L^2(0,T;L^{\frac{6}{5}}(\Gamma_{c
\nu}))}     +   \nu \int\limits_0^T \int\limits_{\Gamma_{c\nu}}|\sigma(u)|^2 \, dy\,ds\rightarrow 0 \ \mbox{ as } \nu\rightarrow 0.
\end{equation}
Then
\begin{equation}\label{main-result-thm-1}
\frac{1}{2}\int\limits_{\mathbb{R}^3}\rho\left|u-u^{E}\right|^2 dy + \frac{1}{\varepsilon^2}\int\limits_{\mc{F}_0}\left(H(\rho)-\frac{\partial H}{\partial\rho}(1)(\rho-1)-H(1)\right)  dy \rightarrow 0 \  \mbox{ as }\varepsilon\rightarrow 0,\ \nu\rightarrow 0.
\end{equation}
\end{theorem}
Theorem~\ref{main-thm} can be compared with two results of Sueur.
For a rigid body moving in an incompressible fluid,
\cite[Theorem~6]{MR2993933} establishes that convergence in the total
kinetic-energy space is equivalent to the vanishing of viscous
dissipation in a boundary strip of thickness proportional to the
viscosity:
\[
    \nu \int_0^T \int_{\Gamma_{c\nu}} |D(u)|^2 \, dx \, dt
    \longrightarrow 0
    \qquad \text{as } \nu \to 0,
    \qquad
    D(u):=\frac12\bigl(\nabla u+\nabla u^\top\bigr).
\]
Our theorem retains this boundary-layer scale while addressing
the simultaneous low Mach number and vanishing viscosity limit.
However, assumption~\eqref{Thm-assum-2} also imposes a weighted
condition on the relative momentum $\rho(u-u_{\mathcal S})$. The convergence concerns the coupled
motion: the total relative kinetic energy controls both the fluid
velocity and the translational and angular velocities of the body.

In the compressible setting,
\cite[Theorem~3]{Sueur-compressible-inviscid} gives a sufficient
boundary-layer condition for convergence from the compressible
Navier--Stokes equations to the compressible Euler equations
in a fixed bounded domain $\Omega$. Writing $\nu$ for the viscosity
parameter and expressing the viscous contribution in our notation,
the condition takes the form
\[
    \nu \int_0^T \int_{\widetilde{\Gamma}_{c\nu}}
    \left(
        \frac{\rho |u|^2}{d_\Omega^2}
        + \frac{\rho^2 |u\cdot n|^2}{d_\Omega^2}
        + |\sigma(u)|^2
    \right) \, dx \, dt
    \longrightarrow 0
    \qquad \text{as } \nu \to 0,
\]
where
\[
    d_\Omega(x):=\operatorname{dist}(x,\partial\Omega),
    \qquad
    \widetilde{\Gamma}_{c\nu}
    :=\{x\in\Omega : d_\Omega(x)<c\nu\},
\]
and $n$ denotes the outward unit normal extended into the strip
along normal rays. This condition combines viscous dissipation
with density-weighted velocity terms near the boundary.
The present theorem considers a different limiting regime and
geometry: the Mach number and viscosity vanish simultaneously,
and the fluid interacts with a freely moving rigid body in an
exterior domain. The initial preparation
condition~\eqref{Thm-Assum-1}, which includes the relative internal
energy scaled by $\varepsilon^{-2}$, is adapted to the limiting
constant fluid density. Together with
\eqref{Thm-assum-2}, it yields convergence to the incompressible
Euler--rigid body system.
\begin{remark}
Assumption~\eqref{Thm-assum-2} consists of two requirements in the
strip $\Gamma_{c\nu}$. The second controls the viscous contribution
near the fluid--solid interface, while the first imposes a weighted
bound on the fluid momentum relative to the body.
\end{remark} 
\section{Proof of the main result}\label{sec4}

Now we want to prove Theorem \ref{main-thm}. We shall establish this result via some intermediate steps. First we  prove a relative energy inequality based on the works by Feireisl and coauthors: \cite{FeJiNo,FeNoSun} for the bounded domains  and \cite{Eduard-Sarka-Sun} for the expanding domains. Here, it is worth mentioning that the relative energy inequality has been established by  Kreml, Ne\v{c}asov\'{a}, Piasecki \cite{KrNePi20} for the case of motion of a  rigid body in the compressible fluids, and  by  Eiter, Ne\v{c}asov\'{a}, Oschmann \cite{ENO}  for the motion of a rotating body in the compressible fluids.

%{\color{blue}  The case with inflow-outflow the weak-strong uniqueness was derived, see \cite{NR}.  \footnote{The incompressible variant of weak-strong uniqueness can be found in the work of Chemetov et al. \cite{CNM}}.}\bigskip 

%This inequality gives  a non-symmetric measure of the difference between a weak solution
%of the c``compressible Navier–Stokes + rigid-body" system and a smooth test function satisfying the no-slip condition. As
%the smooth solutions of the Euler equations do not satisfy the tangential part of the no-slip condition,
%one cannot apply directly these relative energy inequalities. In order to overcome this difficulty we will
%follow a strategy used by Kato in the incompressible setting, see [23], by constructing a “fake” layer
%which, added to the Euler solution, provided a smooth test function which satisfies the no-slip condition
%and yet quite close to the Euler solution.

\subsection{Relative Energy Inequality}\label{Sec-Relative-Noslip}
We will deduce the relative energy inequality in the spirit of \cite{KrNePi20} but we have to establish the inequality for the transformed system \eqref{chg of var mass:fluid}--\eqref{chg of variable boundary:no-slip}, whereas the authors deal with the original system \eqref{mass:fluid}--\eqref{b.c-1}. For a weak solution $(\rho,u)$ to the system \eqref{chg of var mass:fluid}--\eqref{chg of variable boundary:no-slip} and a pair of test functions $(r, U)$, we introduce the following notation for the relative energy:
 \begin{equation}\label{Expression-Relative-energy}
 \mc{E}([\rho,u]\mid [r,U])(\tau)= \frac{1}{2}\int\limits_{\mathbb{R}^3}\rho\left|u-U\right|^2(\tau)\, dy + \frac{1}{\varepsilon^2}\int\limits_{\mc{F}_0}\left(H(\rho)-\frac{\partial H}{\partial\rho}(r)(\rho-r)-H(r)\right)(\tau)\, dy
 \end{equation}
 for a.a $\tau\in (0,T)$, where $H(\rho)$ is defined in \eqref{P-potential}.

 \begin{theorem}\label{rel-energy}
Let $\mc{S}_0 \subset \mathbb{R}^3$ be a $C^{1,1}$ bounded,  closed, simply connected set and $\mc{F}_0= \mathbb{R}^3\setminus \mc{S}_0$. Suppose the pressure is given by \eqref{P-form-new}. Let $(\rho_0,q_0)$ be initial data satisfying \eqref{init}--\eqref{init2} and $(\rho,u)$ be an associated weak solution to system \eqref{chg of var mass:fluid}--\eqref{chg of variable boundary:no-slip} as given by \cref{NS-body}. Then $(\rho,u)$ satisfies the following relative energy inequality:
 \begin{align}\label{Rel-energy}
 \mc{E}\left([\rho,u]\mid [r,U]\right)(t) +\nu \int\limits_0^{t}\int\limits_{\mc{F}_0} \sigma(u) :\mathbb D (u) \, dy \, ds
 \leq \mc{E}\left([\rho_0,q_0/\rho_0]\mid [r(0),U(0)]\right) + \int\limits_0^{t} \mc{R}(\rho,u,U)(s)\, ds,
 \end{align}
 for a.a. $t\in(0,T)$ and for any
 \begin{equation}\label{rel:test}
 U \in C^\infty([0,T]; \mc H  \cap C^\infty_c(\mathbb R^3) ) \, \text{ and } \, r\in C^\infty([0,T] \times  \mathbb R^3) . 
 \end{equation}
 %  U \in H^1([0,T]; \mc H ) \cap L^\infty( 0,T; \mc{V})
 The remainder term $\mc{R}$ is given by
\begin{multline}\label{Remainder-term}
\mathcal R(\rho, u, U)(t) = 
\nu  \int\limits_{\mc{F}_0}   \sigma(u) :  \mathbb D(U) \, dy  +  \int\limits_{\mc F_0}\left[ \rho \frac{\partial U}{\partial s}\cdot(U-u) + \rho(u-u_{\mc{S}}) \cdot \mathbb D(U)  \cdot (U-u) \right]  
\, dy \\ 
+  m (\ell_{U}-\ell)\cdot \ell'_U  + J_0 (\omega_U -\omega)\cdot \omega'_U
+ \int\limits_{\mc{F}_{0}} \operatorname{det} (\rho \omega, u, U)\,dy - 
\operatorname{det}(m \ell, \omega,\ell_{U}) 
 -  \operatorname{det}(J_{0}\omega,\omega,\omega_{U}) \\ -  \frac{1}{\veps^2}\int\limits_{\mc{F}_0}  p(\rho) \operatorname{div} U \, dy
 + \frac{1}{\veps^2}\int\limits_{\mathbb R^3} (r-\rho) \frac{\partial H^\prime(r)}{\partial s} \,  dy
 - \frac{1}{\veps^2}\int\limits_{\mathbb R^3} \rho(u-u_{\mc{S}}) \cdot \nabla H'(r) \,  dy \\
 + \frac{1}{\veps^2} \int\limits_{\mc S_0} \bigg( \frac{\partial H}{\partial \rho}(r)(\rho_{\mc S}-r) + H(r)\bigg) (t,\cdot) \, dy - \frac{1}{\veps^2}\int\limits_{\mc S_0} \frac{\partial H}{\partial \rho}(r(0,\cdot)) (\rho_{\mc S} - r(0,\cdot)) -H(r)(0,\cdot) \bigg) \, dy ,
\end{multline}
 where $U(t,y)= \ell_{U}(t)
+ \omega_{U}(t) \times y   \quad \mbox{ for all }y \in \mc{S}_{0}.$ 
 \end{theorem}
\begin{proof} 
 We begin with testing the momentum equation \eqref{weakform-momentum} by $U$, resulting the following relation:
\begin{multline}\label{weakform-momentum-U}
 \int\limits_{\mc{F}_{0}} \rho u(t,y)\cdot U(t,y)\,dy + m\ell(t)\cdot \ell_{U}(t)
 + J_{0}\omega(t)\cdot \omega_{U}(t) - \int\limits_{\mc{F}_{0}} q_0 \cdot U(0,y)\,dy - m\ell_0\cdot \ell_{U}(0) - J_{0} \omega_0 \cdot \omega_{U}(0) \\ =  \int\limits_{0}^{t} \int\limits_{\mc{F}_{0}}\rho u \cdot \frac{\partial U}{\partial s}\,dy\,ds + \int\limits_{0}^{t}m\ell\cdot \ell'_{U}\,ds
+ \int\limits_{0}^{t} J_{0}\omega\cdot \omega'_{U}\,ds +\int\limits_{0}^{t}\int\limits_{\mc{F}_{0}} \left[\rho(u-u_{\mc{S}}) \otimes u \right] : \mathbb D (U)  \,dy\,ds + \int\limits_{0}^{t}
\operatorname{det}(m \ell,\omega,\ell_{U})\,ds \\  - \int\limits_{0}^{t}\int\limits_{\mc{F}_{0}} \operatorname{det} (\rho \omega, u, U)\,dy\,ds 
 + \int\limits_{0}^{t} \operatorname{det}(J_{0}\omega,\omega,\omega_{U}) \,ds -  \int\limits_{0}^{t}\int\limits_{\mc{F}_{0}} \left(\nu \sigma(u) - \frac{1}{\varepsilon^2}p(\rho)\mathbb{I}\right): \mathbb D (U) \,dy\,ds .
 \end{multline}
Next, we test the continuity equation \eqref{weakform-continuity} by $\frac{1}{2} |U|^2$ which yields
\begin{equation}\label{weakform-continuity-U} \int\limits_{0}^{t}\int\limits_{\mathbb{R}^3} \left[\rho U \cdot \frac{\partial U}{\partial s}  + [\rho (u-u_{\mc{S}}) \cdot \mathbb D(U)] \cdot U  \right]\,dy\,ds = \frac{1}{2}\int\limits_{\mathbb R^3} \rho(t,y)|U(t,y)|^2 \, dy  - 
\frac{1}{2}\int\limits_{\mathbb{R}^3}  \rho_0 |U(0,y)|^2 \, dy    .
\end{equation} 
Summing up the energy inequality \eqref{energy} and \eqref{weakform-continuity-U}, we obtain 
\begin{multline}\label{Energy-interm}
 \frac{1}{2}\int\limits_{\mathbb R^3} \rho\left(|u|^2 + |U|^2  \right)(t,\cdot) \, dy + \frac{1}{\veps^2} \int\limits_{\mc{F}_0} H(\rho)(t,\cdot)\, dy+ \nu \int\limits_0^t\int\limits_{\mc{F}_0} \sigma(u): \mathbb D(u) \, dy \, ds \\
 \leq  \int\limits_{0}^{t}\int\limits_{\mathbb{R}^3} \left[\rho U \cdot \frac{\partial U}{\partial s}  + [\rho (u-u_{\mc{S}}) \cdot \mathbb D(U) ] \cdot U  \right] dy\,ds 
 + \int\limits_{\mathbb R^3} \left(\frac{|q_0|^2}{2\rho_0} + \frac{|\rho_0 U(0,y)|^2}{2\rho_0}\right)  \mathds{1}_{\{\rho_0>0\}} \, d y  + 
\frac{1}{\veps^2} \int\limits_{\mc{F}_0} H(\rho_0)\, dy .
\end{multline}
Recall that by definition (see \eqref{convention}) and the fact that $u=u_{\mc S}$ on $\mc S_0$, we have 
\begin{align*}
&\int\limits_0^t \int\limits_{\mathbb{R}^3} \left[\rho U \cdot \frac{\partial U}{\partial s}  + [\rho (u-u_{\mc{S}}) \cdot \mathbb D(U) ] \cdot U  \right] dy\,ds \\ 
&= \int\limits_{0}^{t}\int\limits_{\mc F_0} \rho U \cdot \frac{\partial U}{\partial s} \, dy \, ds + \int\limits_0^t \left( m \ell_U \cdot \ell'_U + J_0 \omega_U \cdot \omega'_U \right) \, ds +  \int\limits_{0}^{t}\int\limits_{\mc F_0} [\rho (u-u_{\mc{S}}) \cdot \mathbb D(U) ] \cdot U \, dy\,ds.
\end{align*}
Then, subtracting the relation \eqref{weakform-momentum-U} from \eqref{Energy-interm} we get for a.a.  $t\in (0,T)$:
\begin{multline}\label{Energy-Interm-2}
\frac{1}{2}\int\limits_{\mathbb R^3} 
\rho\left|u - U\right|^2(t,\cdot) \, dy + \frac{1}{\veps^2} \int\limits_{\mc{F}_0} H(\rho)(t,\cdot)\, dy
+\nu \int\limits_0^t \int\limits_{\mc{F}_0} \sigma(u): \mathbb D(u) \, dy \, ds\\ \leq \int\limits_{\R^3}  \frac{1}{2\rho_0}\left|q_0 - \rho_0 U(0,y)\right|^2 \mathds{1}_{\{\rho_0>0\}}\, dy 
+  \frac{1}{\veps^2} \int\limits_{\mc{F}_0} H(\rho_0) \, dy 
+ \nu \int\limits_0^t \int\limits_{\mc{F}_0} \sigma(u) :  \mathbb D(U) \, dy \, ds -\frac{1}{\veps^2} \int\limits_0^t \int\limits_{\mc{F}_0} p(\rho) \mathbb I : \mathbb D(U) \, dy\, ds \\
 +\int\limits_0^t \int\limits_{\mc{F}_0}\left[
\rho \frac{\partial U}{\partial s}\cdot (U-u)
+ \rho(u-u_{\mc{S}}) \cdot \mathbb D(U) \cdot (U-u) \right] dy\, ds 
+ \int\limits_0^t  m(  \ell_U-\ell)\cdot \ell'_U  \, ds 
 + \int\limits_0^t  J_0 (\omega_U - \omega)\cdot\omega'_U  \, ds  \\
+ \int\limits_{0}^{t}\int\limits_{\mc{F}_{0}} \operatorname{det} (\rho \omega, u, U)\,dy\,ds - \int\limits_{0}^{t}
\operatorname{det}(m \ell,\omega,\ell_{U})\,ds
 - \int\limits_{0}^{t} \operatorname{det}(J_{0}\omega,\omega,\omega_{U}) \,ds. 
\end{multline}
Finally,  we test the continuity equation in   with $H^\prime(r)$ for  $r\in C^\infty([0,T]\times \R^3)$, that yields
\begin{align}\label{continuity-H-1} \int\limits_{0}^{t}\int\limits_{\mathbb{R}^3} \left[\rho \frac{\partial H^\prime(r)}{\partial s} + \rho (u-u_{\mc{S}})\cdot  \nabla H^\prime(r)\right] dy\,ds = \int\limits_{\mathbb R^3} (\rho H'(r))(t,\cdot) \, dy  - 
\int\limits_{\mathbb{R}^3}  \rho_0 H'(r)(0,\cdot) \, dy   .
\end{align} 
Using the fundamental  theorem of calculus, we further have
\begin{align}\label{continuity-H-2}
    \int\limits_0^t \int\limits_{\mathbb R^3} \frac{\partial p(r)}{\partial s} \, dy\, ds = \int\limits_{\mathbb R^3} p(r)(t,\cdot) \, dy -  \int\limits_{\mathbb R^3} p(r)(0,\cdot) \, dy 
\end{align}
Again, from the relation 
\begin{align*}
H'(r)r - H(r) = p(r) ,
\end{align*}
a direct manipulation gives 
\begin{align}\label{continuity-H-3}
\int\limits_0^t \int\limits_{\mathbb R^3}  r \frac{\partial H'(r)}{\partial s} \, dy\, ds = \int\limits_0^t \int\limits_{\mathbb R^3} \frac{\partial p(r)}{\partial s} \, dy \, ds.
\end{align}

Combining  \eqref{continuity-H-1}--\eqref{continuity-H-3} with \eqref{Energy-Interm-2}, and writing in a more concrete form, we have  
\begin{align}\label{Energy-Interm-3}
&\frac{1}{2}\int\limits_{\mathbb R^3} 
\rho\left|u - U\right|^2(t,\cdot) \, dy + \frac{1}{\veps^2} \int\limits_{\mc{F}_0}\left( H(\rho) -  \frac{\partial H}{\partial \rho}(r)(\rho - r) - H(r) 
\right) 
(t,\cdot)\, dy 
+\nu \int\limits_0^t \int\limits_{\mc{F}_0} \sigma(u) : \mathbb D(u) \, dy \, ds  \notag \\
&  \leq 
\int\limits_{\R^3}  \frac{1}{2\rho_0}\left|q_0 - \rho_0 U(y,0)\right|^2 \mathds{1}_{\{\rho_0>0\}}\, dy 
+  \frac{1}{\veps^2} \int\limits_{\mc{F}_0} \left( H(\rho_0) - \frac{\partial H}{\partial \rho}(r(0,\cdot))\left(\rho_0-r(0,\cdot)\right) - H(r)(0,\cdot) \right) dy \notag  \\
& + \nu \int\limits_0^t \int\limits_{\mc{F}_0}   \sigma(u) :  \mathbb D( U ) \, dy\, ds + \int\limits_0^t \int\limits_{\mc F_0} \left[  \rho \frac{\partial U}{\partial s}\cdot (U-u)  + \rho(u-u_{\mc{S}}) \cdot \mathbb D(U)  \cdot (U-u) \right]  
dy \, ds  \notag \\
 & + \int\limits_0^t m (\ell_U - \ell) \cdot  \ell'_{U} \, ds + \int\limits_0^t J_0 (\omega_U - \omega) \cdot \omega'_U \, ds 
 \notag  \\
&+ \int\limits_{0}^{t}\int\limits_{\mc{F}_{0}} \operatorname{det} (\rho \omega, u, U)\,dy\,ds - \int\limits_{0}^{t}
\operatorname{det}(m \ell, \omega,\ell_{U})\,ds
 - \int\limits_{0}^{t} \operatorname{det}(J_{0}\omega,\omega,\omega_{U}) \,ds  \notag \\
 &  + 
  \frac{1}{\veps^2}\int\limits_0^t\int\limits_{\mathbb R^3} (r-\rho) \frac{\partial H^\prime(r)}{\partial s} \, dy\, ds -  \frac{1}{\veps^2}\int\limits_0^t\int\limits_{\mc{F}_0}  p(\rho) \operatorname{div} U \,dy\, ds   
 - \frac{1}{\veps^2}\int\limits_0^t\int\limits_{\mathbb R^3} \rho(u-u_{\mc{S}}) \cdot \nabla H'(r) \, dy\, ds \notag  \\
 & + \frac{1}{\veps^2} \int\limits_{\mc S_0} \bigg( \frac{\partial H}{\partial \rho}(r)(\rho_{\mc S}-r) + H(r)\bigg) (t,\cdot) \, dy - \frac{1}{\veps^2}\int\limits_{\mc S_0} \frac{\partial H}{\partial \rho}(r(0,\cdot)) (\rho_{\mc S} - r(0,\cdot)) -H(r)(0,\cdot) \bigg) \, dy 
\end{align} 
for a.a.  $t\in (0,T)$. 

This ends the proof of  the relative energy inequality \eqref{Rel-energy}. 
\end{proof}

\begin{cor}[\textbf{Uniform bounds}] \label{Corolary-Noslip}
Suppose that the  assumptions of Theorem \ref{rel-energy} hold true. Then, a weak solutions to system \eqref{chg of var mass:fluid}--\eqref{chg of variable boundary:no-slip}
satisfy  the following  bounds: 
\begin{equation}\label{Uniform-bound-compressible}
\begin{aligned}
&\esssup_{t\in [0,T]} \| \sqrt{\rho} u (t, \cdot)  \|_{L^2(\F_0; \mathbb R^3)} \leq C , \\
&\esssup_{t\in [0,T]} \left\| \bigg[\frac{\rho-1}{\veps}\bigg]_{\ess} (t, \cdot) \right\|_{L^2(\mc F_0)} \leq C , \\
&\esssup_{t\in [0,T]} \left(  \| [\rho^\gamma]_{\res}(t, \cdot)\|_{L^1(\F_0)} + \|[1]_{\res}(t, \cdot)\|_{L^1(\F_0)}  \right) \leq \veps^2 C  , \\
& \ \  \sqrt{\nu} \| u \|_{L^2(0,T; W^{1,2}(\mc F_0;\mathbb R^3))} \leq C ,
\end{aligned}
\end{equation}
where the constant $C>0$ is independent of $\veps$ or $\nu$. In particular, $\|u(t)\|_{\mc H}$ is bounded uniformly in $\nu, \veps$ for almost all $t\in [0,T]$. 
\end{cor}

\begin{proof}
Considering   $U=0$ and $r=1$ in the estimate \eqref{Energy-Interm-3}, we have 
\begin{align}\label{Estimate-frm-Rel-En}
&\frac{1}{2}\int\limits_{\R^3} \rho |u|^2 (\cdot,t)  + \frac{1}{\veps^2} \int\limits_{\mc F_0} \left( H(\rho) - \frac{\partial H}{\partial \rho}(1) (\rho-1) - H(1) \right) (\cdot, t) \, dy  + \nu \int\limits_0^t \int\limits_{\mc F_0} \sigma(u) : \mathbb D(u) \, dy \, ds \notag \\
&\leq \int\limits_{\R^3} \frac{|q_0|^2}{2\rho_0}  + \frac{1}{\veps^2} \int\limits_{\mc F_0} \left( H(\rho_0) - \frac{\partial H}{\partial \rho}(1) (\rho_0-1)  - H(1)  \right) \, dy ,
\end{align}
which is uniform in $\veps$ or $\nu$, thanks to the assumption \eqref{Thm-Assum-1}.

We have the following coercivity properties: for any compact set $K\subset (0,\infty)$, there exists some constant $c(K)$ such that 
\begin{align*}
& H(\rho) - \frac{\partial H}{\partial \rho}(1)(\rho-1) - H(1) \geq c(K) \| \rho - 1 \|_{L^2(\mc F_0)}  \ \text{ if } \, \rho \in K ,\\
& H(\rho) - \frac{\partial H}{\partial \rho}(1)(\rho-1) - H(1) \geq c(K) (1 + \rho^\gamma )  \ \text{ if } \, \rho \in (0,\infty)\setminus K ,
\end{align*}
see for instance \cite[Section 4]{Feireisl-Novotny-INVISCID} (in the context of Navier-Stokes-Fourier system). Now, 
it is convenient to introduce the following decomposition (see \cite[Section 4]{Feireisl-Novotny-INVISCID}):  
\begin{align*}
\Sigma = [\Sigma]_{\ess} + [\Sigma]_{\res} , 
\end{align*}
where $[\Sigma]_{\ess} =  \Sigma \mathds{1}_{\{1/2\leq \rho \leq 2\}}, \ \  [\Sigma]_{\res} =  \Sigma - [\Sigma]_{\ess}$. Consequently, using the estimate \eqref{Estimate-frm-Rel-En} and the assumption \eqref{Thm-Assum-1}, we deduce the bounds in \eqref{Uniform-bound-compressible}. In particular, the bound $\eqref{Uniform-bound-compressible}_4$ is followed from the generalized Korn's inequality
\begin{align*}
\left\| \nabla w + \nabla^\top w - \frac{2}{3} \Div w  \right\|_{L^2(\mathbb R^3, \mathbb R^{3\times 3})} \geq C \|\nabla w \|_{L^2(\mathbb R^3;\mathbb R^{3\times 3})} , \ \ \ \text{for } \,  w \in W^{1,2}(\R^3; \R^3) , 
\end{align*}
and the estimate \eqref{Estimate-frm-Rel-En}. 
\end{proof}

%Moreover, the uniform boundedness of the right-hand side of \eqref{Uniform-bound-compressible} holds good due to  the assumption \eqref{Thm-Assum-1}. 

\subsection{A Kato type ``fake" layer}

We cannot use the solution  $u^{E}$ of \eqref{chg of var momentum:fluidE}--\eqref{chg of var initial condE} as a test function in the balance of momentum equation \eqref{weakform-momentum} (corresponds to Navier--Stokes system) or in the relative energy inequality \eqref{Rel-energy}. The reason is that $u^{E}$ has $H^1$-space regularity only in the fluid domain $\mc{F}_0$ whereas the requirement for the test function in \eqref{weakform-momentum} or in \eqref{Rel-energy} is the $H^1$ regularity for the whole $\mathbb{R}^3$. That means, we need to modify $u^{E}$ such that there is regular $H^1$ transition through the interaction boundary $\partial \mc{S}_0$. In order to overcome this difficulty we construct a ``fake" layer. If we add this layer to the Euler solution, it will give a suitable test function and yet quite close to the Euler solution.  Let us recall the following result from  \cite[Proposition 2]{MR2993933}.

{
\begin{proposition}\label{Proposition-Fake-Layer}
Let $\mc{S}_0\subset\mathbb R^3$ be a bounded, closed, simply
connected set with boundary of class $C^{1,1}$, and let
$\mc{F}_0:=\mathbb R^3\setminus\mc{S}_0.$
Assume that
$
(u^E,p^E,\ell^E,\omega^E)
$
is the strong solution to
\eqref{chg of var momentum:fluidE}--\eqref{chg of var initial condE}
corresponding to initial data
$
(u_0^E,\ell_0^E,\omega_0^E)
$
as in \cref{Thm-Euler-body}.
Fix $c>0$. Then there exists $\nu_0>0$ such that, for every
$0<\nu\leq\nu_0$, there exists a vector field
$
v^F
$
with the following properties:

\begin{enumerate}

\item\label{fake-layer-structure}
The vector field $v^F$ satisfies
$
v^F\in C^1([0,T];\mc{H}^E)$,
$v^F=0
\quad\text{in }\mc{S}_0,
$
and
$
\operatorname{div}v^F=0
\quad\text{in }\mathcal D'(\mathbb R^3).
$

\item\label{fake-layer-support}
The vector field is supported in the boundary layer:
$
\operatorname{supp}v^F
\subset
\overline{\Gamma_{c\nu}}.
$
Moreover, it agrees with the relative Euler velocity in a smaller
boundary layer: $v^F=u^E-u_{\mc S}^E
\quad\text{in }\Gamma_{c\nu/2}$, where $u_{\mc S}^E(t,y)
:=
\ell^E(t)+\omega^E(t)\times y.$

\item\label{fake-layer-zero-order}
For every $1\leq p\leq\infty$ and every $\alpha\geq0$,
\begin{equation*}
\left\|d^\alpha v^F\right\|_
{L^\infty(0,T;L^p(\mathbb R^3))}
+
\left\|d^\alpha\partial_tv^F\right\|_
{L^\infty(0,T;L^p(\mathbb R^3))}
\leq
C_{\alpha,p}\nu^{\alpha+1/p}.
\end{equation*}

In particular,
\[
\|v^F\|_{L^\infty((0,T)\times\mathbb R^3)}
+
\|\partial_tv^F\|_
{L^\infty((0,T)\times\mathbb R^3)}
\leq C,
\]
and
\[
\|v^F\|_{C([0,T];\mc{H}^E)}
+
\|\partial_tv^F\|_{C([0,T];\mc{H}^E)}
\leq C\nu^{1/2}.
\]

\item\label{fake-layer-gradient}
For every $1\leq p\leq\infty$ and every $\alpha\geq0$,
\begin{equation*}
    \left\|d^\alpha\nabla v^F\right\|_
{L^\infty(0,T;L^p(\mc{F}_0))}
\leq
C_{\alpha,p}\nu^{\alpha-1+1/p}.
\end{equation*}
In particular,
$\|\nabla v^F\|_
{L^\infty(0,T;L^2(\mc{F}_0))}
\leq
C\nu^{-1/2},$  
$\|d\nabla v^F\|_
{L^\infty(0,T;L^2(\mc{F}_0))}
\leq
C\nu^{1/2},$ $\|d\nabla v^F\|_
{L^\infty((0,T)\times\mc{F}_0)}
\leq C,$ 
and $\|d^2\nabla v^F\|_
{L^\infty((0,T)\times\mc{F}_0)}
\leq
C\nu.$

\item\label{fake-layer-corrected-Euler}
The corrected Euler velocity satisfies
$
u^E-v^F
\in
C([0,T];\mc{H}^E)
\cap
L^2(0,T;\mc{V}).
$

\end{enumerate}

The constants $C_{\alpha,p}$ are independent of $\nu$, but may
depend on $c$, $\alpha$, $p$, $T$, the geometry of $\mc{S}_0$, and
the strong Euler solution.

%In \cref{fake-layer-gradient}, $\nabla v^F$ denotes the gradient in$\mc{F}_0$. If it is extended by zero to $\mc{S}_0$, it is understoodas the piecewise gradient and not as the distributional gradient ofthe zero extension.
\end{proposition}
\begin{proof}
Let us set $h:=c\nu.$

\medskip
\noindent\textbf{Step 1: Geometry of the boundary layer.}

Since \(\partial\mc{S}_0\) is of class \(C^{1,1}\), there exists
\(\delta_0>0\) such that every point
$
y\in\mc{F}_0,
\quad
0<d(y)<\delta_0,
$
has a unique nearest point
$
\pi(y)\in\partial\mc{S}_0.
$
Moreover, in the exterior tubular neighbourhood
\[
\mathcal U_{\delta_0}
:=
\left\{
 y\in\mc{F}_0:\ 0<d(y)<\delta_0
\right\},
\]
the distance function \(d\) belongs to \(C^{1,1}\) and satisfies
\[
y=\pi(y)+d(y)n(\pi(y)),
\qquad
|\nabla d(y)|=1
\]
for almost every \(y\in\mathcal U_{\delta_0}\). In addition,
$
\|\nabla^2d\|_{L^\infty(\mathcal U_{\delta_0})}
\leq C.
$
Here and below, the second derivatives of \(d\) are understood almost
everywhere. Choose \(\nu_0>0\) sufficiently small so that
$
c\nu_0<\min\{\delta_0,1\}.
$
Thus, for \(0<\nu\leq\nu_0\), the boundary layer
\(\Gamma_h=\Gamma_{c\nu}\) is contained in
\(\mathcal U_{\delta_0}\).

We next recall the measure estimate for \(\Gamma_h\). Using the normal
coordinates
\[
\Phi(q,r):=q+r n(q),
\qquad
q\in\partial\mc{S}_0,
\quad
0<r<h,
\]
we have
\[
|\Gamma_h|
=
\int_{\partial\mc{S}_0}
\int_0^h J(q,r)\,dr\,dS(q),
\]
where \(J(q,r)\) is the Jacobian of the normal-coordinate map. Since
\(\partial\mc{S}_0\) is compact and of class \(C^{1,1}\), this
Jacobian is uniformly bounded for \(0<r<\delta_0\). Consequently,
\begin{equation}\label{eq:fake-layer-measure}
|\Gamma_h|
\leq Ch
=
Cc\nu.
\end{equation}

\medskip
\noindent\textbf{Step 2: Construction and estimates for the
antisymmetric potential.}

Set
$
w:=u^E-u_S^E.
$ Since both $u^E$ and $u_S^E$ are divergence-free and the Euler
impermeability condition gives
$
u^E\cdot n=u_S^E\cdot n
        \quad\text{on }\partial\mathcal S_0,
$
we have
$
\operatorname{div}w=0\quad\text{in }\mathcal F_0,
\quad
w\cdot n=0\quad\text{on }\partial\mathcal S_0.
$
By \cite[Lemma A.1]{Kato-zero-viscosity}, there exists an
antisymmetric tensor field
$
a_F=(a_{F,ij})_{1\leq i,j\leq3},
\qquad
a_{F,ij}=-a_{F,ji},
$
such that, in the fixed tubular neighbourhood
\(\mathcal U_{\delta_0}\),
\begin{equation}\label{eq:fake-layer-potential}
\partial_j a_{F,ij}=w_i,
\qquad
a_F=0
\quad\text{on }\partial\mc{S}_0.
\end{equation}
Here and below, repeated indices are summed over.

The regularity of the strong Euler solution, together with the
construction of \(a_F\), gives
\begin{align}
\sup_{t\in[0,T]}
\Big(
 \|w(t)\|_{W^{1,\infty}(\mathcal U_{\delta_0})}
 +\|\partial_tw(t)\|_{L^\infty(\mathcal U_{\delta_0})}
 +\|a_F(t)\|_{W^{1,\infty}(\mathcal U_{\delta_0})}
 +\|\partial_ta_F(t)\|_{W^{1,\infty}(\mathcal U_{\delta_0})}
\Big)
\leq C.
\label{eq:fake-layer-regularity}
\end{align}
The corresponding mappings are continuous in time with values in the
indicated spaces.

Let \(y\in\Gamma_h\), and let \(\pi(y)\) be the nearest point to \(y\)
on \(\partial\mc{S}_0\). Since $a_F(t,\pi(y))=0,$ the fundamental theorem of calculus, applied componentwise along the
segment joining \(\pi(y)\) to \(y\), gives
\begin{align*}
a_{F,ij}(t,y)
=
a_{F,ij}(t,y)-a_{F,ij}(t,\pi(y))
=
\int_0^1
\nabla a_{F,ij}
\bigl(t,\pi(y)+s(y-\pi(y))\bigr)
\mathbin{\cdot}(y-\pi(y))\,ds.
\end{align*}
It follows from \eqref{eq:fake-layer-regularity} that
\begin{equation}\label{eq:fake-layer-a-distance}
|a_F(t,y)|
\leq
C|y-\pi(y)|
=
Cd(y).
\end{equation}
Applying the same componentwise argument to \(\partial_ta_F\), we
obtain
\begin{equation}\label{eq:fake-layer-at-distance}
|\partial_ta_F(t,y)|
\leq
Cd(y),
\qquad
(t,y)\in[0,T]\times\Gamma_h.
\end{equation}

\medskip
\noindent\textbf{Step 3: Definition of the fake boundary layer.}

Choose a function \(\xi\in C^\infty([0,\infty))\) such that
\[
0\leq\xi\leq1,
\qquad
\xi(r)=1
\quad\text{for }0\leq r\leq\frac12,
\qquad
\xi(r)=0
\quad\text{for }r\geq1.
\]
Define
\[
z_h(y)
:=
\xi\left(\frac{d(y)}{h}\right),
\qquad
y\in\mc{F}_0.
\]

We define an antisymmetric tensor \(A_h\) on \(\mathbb R^3\) by
\[
A_h(t,y)
:=
\begin{cases}
z_h(y)a_F(t,y),
&
y\in\mc{F}_0,
\\[1mm]
0,
&
y\in\mc{S}_0.
\end{cases}
\]
Since \(a_F=0\) on \(\partial\mc{S}_0\), the tensor
\(z_ha_F\) has zero trace on \(\partial\mc{S}_0\). We now define
\[
v^F:=\operatorname{div}A_h
\qquad
\text{in }\mathcal D'(\mathbb R^3).
\]
Thus,
\begin{equation}\label{eq:fake-layer-definition}
v^F(t,y)
=
\begin{cases}
\operatorname{div}(z_ha_F)(t,y),
&
y\in\mc{F}_0,
\\[1mm]
0,
&
y\in\mc{S}_0.
\end{cases}
\end{equation}
Using \eqref{eq:fake-layer-potential} and the product rule, we obtain,
in \(\mc{F}_0\),
\begin{align}\label{eq:fake-layer-formula}
v_i^F
=
\partial_j(z_ha_{F,ij})
=
z_h\partial_ja_{F,ij}
+
(\partial_jz_h)a_{F,ij}
=
z_hw_i
+
\frac1h
\xi'\left(\frac{d}{h}\right)
a_{F,ij}\partial_jd.
\end{align}

Since \(\xi(r)=0\) for \(r\geq1\), both \(z_h\) and its derivatives
vanish when \(d\geq h\). Therefore,
\[
\operatorname{supp}v^F
\subset
\overline{\Gamma_h}
=
\overline{\Gamma_{c\nu}}.
\]

Moreover, \(A_h\) is antisymmetric. Hence, 
\begin{align*}
\operatorname{div}v^F
=
\partial_i\partial_jA_{h,ij}
=
\frac12
\partial_i\partial_j
\bigl(A_{h,ij}+A_{h,ji}\bigr)
=0.
\end{align*}
Since \(v^F=0\) in \(\mc{S}_0\), it is a rigid vector field in the
solid, with zero translational and angular velocities. Thus,
\[
v^F(t)\in\mc{H}^E
\qquad
\text{for every }t\in[0,T].
\]
The time regularity in
\eqref{eq:fake-layer-regularity} also shows that
$
t\longmapsto v^F(t)
$
is continuously differentiable with values in \(L^2(\mathbb R^3)\).
Both \(v^F\) and \(\partial_tv^F\) are divergence-free and vanish in
\(\mc{S}_0\). Consequently,
\begin{equation}\label{eq:fake-layer-C1H}
v^F\in C^1([0,T];\mc{H}^E).
\end{equation}

Finally, if \(d(y)<h/2\), then
$
z_h(y)=1,
\quad
\nabla z_h(y)=0.
$
It follows from \eqref{eq:fake-layer-definition} and
\eqref{eq:fake-layer-potential} that
\begin{equation}\label{eq:fake-layer-matching}
v^F
=
\operatorname{div}a_F
=
u^E-u_{\mc S}^E
\qquad
\text{in }\Gamma_{h/2}
=
\Gamma_{c\nu/2}.
\end{equation}
This proves the structural and localization properties in
(\ref{fake-layer-structure}) and (\ref{fake-layer-support}).

\medskip
\noindent\textbf{Step 4: Pointwise estimates for \(v^F\) and
\(\partial_tv^F\).}

On the support of \(\xi'(d/h)\), one has
$
0\leq d(y)\leq h.
$
Using \eqref{eq:fake-layer-regularity},
\eqref{eq:fake-layer-a-distance}, and \(|\nabla d|=1\), we obtain
\begin{align*}
\left|
\frac1h
\xi'\left(\frac{d}{h}\right)
a_F\nabla d
\right|
\leq
\frac{C}{h}|a_F|
\leq
C\frac{d}{h}
\leq C.
\end{align*}
Consequently, formula \eqref{eq:fake-layer-formula} gives
\begin{equation}\label{eq:fake-layer-v-Linfty}
|v^F(t,y)|\leq C
\qquad
\text{for a.e. }(t,y)\in(0,T)\times\mathbb R^3.
\end{equation}

Since \(d\) and \(z_h\) are independent of time, differentiating
\eqref{eq:fake-layer-formula} gives
\begin{equation}\label{eq:fake-layer-dtv-formula}
\partial_tv_i^F
=
z_h\partial_tw_i
+
\frac1h
\xi'\left(\frac{d}{h}\right)
\partial_ta_{F,ij}\partial_jd.
\end{equation}
Using \eqref{eq:fake-layer-regularity},
\eqref{eq:fake-layer-at-distance}, and \(d\leq h\), we obtain
\begin{align*}
\left|
\frac1h
\xi'\left(\frac{d}{h}\right)
\partial_ta_F\nabla d
\right|
\leq
\frac{C}{h}|\partial_ta_F|
\leq
C\frac{d}{h}
\leq C.
\end{align*}
Therefore,
\begin{equation}\label{eq:fake-layer-dtv-Linfty}
|\partial_tv^F(t,y)|
\leq C
\quad
\text{for a.e. }(t,y)\in(0,T)\times\mathbb R^3.
\end{equation}

Let \(1\leq p<\infty\) and \(\alpha\geq0\). Since \(v^F\) is
supported in \(\Gamma_h\), \(d\leq h\) on its support, and
\(|v^F|\leq C\), we have
\begin{align*}
\|d^\alpha v^F(t)\|_{L^p(\mathbb R^3)}^p
=
\int_{\Gamma_h}d^{\alpha p}|v^F(t,y)|^p\,dy
\leq
Ch^{\alpha p}|\Gamma_h|.
\end{align*}
Using \eqref{eq:fake-layer-measure}, we obtain
\begin{equation}\label{eq:fake-layer-weighted-v}
\|d^\alpha v^F(t)\|_{L^p(\mathbb R^3)}
\leq
C_{\alpha,p}h^{\alpha+1/p}.
\end{equation}
For \(p=\infty\), the same conclusion follows directly from
\eqref{eq:fake-layer-v-Linfty} and \(d\leq h\):
\begin{equation}\label{eq:fake-layer-weighted-v-infty}
\|d^\alpha v^F(t)\|_{L^\infty(\mathbb R^3)}
\leq
C_\alpha h^\alpha.
\end{equation}
The same argument, using
\eqref{eq:fake-layer-dtv-Linfty}, gives for every \(1\leq p\leq\infty\),
\begin{equation}\label{eq:fake-layer-weighted-dtv}
\|d^\alpha\partial_tv^F(t)\|_{L^p(\mathbb R^3)}
\leq
C_{\alpha,p}h^{\alpha+1/p}.
\end{equation}
Since \(h=c\nu\), estimates
\eqref{eq:fake-layer-weighted-v}--\eqref{eq:fake-layer-weighted-dtv}
yield
\[
\left\|d^\alpha v^F\right\|_
{L^\infty(0,T;L^p(\mathbb R^3))}
+
\left\|d^\alpha\partial_tv^F\right\|_
{L^\infty(0,T;L^p(\mathbb R^3))}
\leq
C_{\alpha,p}\nu^{\alpha+1/p}.
\]
In particular, choosing \(\alpha=0\) and \(p=2\), we find
\[
\|v^F(t)\|_{L^2(\mathbb R^3)}
+
\|\partial_tv^F(t)\|_{L^2(\mathbb R^3)}
\leq
C\nu^{1/2}.
\]
Since \(v^F\) and \(\partial_tv^F\) vanish in \(\mc{S}_0\), their
translational and angular rigid components are zero. Therefore, their
\(\mc{H}^E\)-norms are controlled by their
\(L^2(\mathbb R^3)\)-norms. Together with
\eqref{eq:fake-layer-C1H}, this gives
\[
\|v^F\|_{C([0,T];\mc{H}^E)}
+
\|\partial_tv^F\|_{C([0,T];\mc{H}^E)}
\leq
C\nu^{1/2}.
\]
This proves (\ref{fake-layer-zero-order}).

\medskip
\noindent\textbf{Step 5: Estimates for the spatial gradient.}

We differentiate \eqref{eq:fake-layer-formula} with respect to
\(y_k\). Using
\[
\partial_kz_h
=
\frac1h
\xi'\left(\frac{d}{h}\right)\partial_kd,
\]
we obtain
\begin{multline}\label{eq:fake-layer-gradient-formula}
\partial_kv_i^F
=
z_h\partial_kw_i
+
\frac1h
\xi'\left(\frac{d}{h}\right)
(\partial_kd)w_i
+
\frac1h
\xi'\left(\frac{d}{h}\right)
(\partial_ka_{F,ij})(\partial_jd)
+
\frac1h
\xi'\left(\frac{d}{h}\right)
a_{F,ij}\partial_{kj}d \\
+
\frac1{h^2}
\xi''\left(\frac{d}{h}\right)
a_{F,ij}(\partial_kd)(\partial_jd).
\end{multline}
We estimate the five terms in
\eqref{eq:fake-layer-gradient-formula}. From
\eqref{eq:fake-layer-regularity},
$|w|+|\nabla w|+|\nabla a_F|\leq C.$
Furthermore,
\[
|\nabla d|\leq1,
\qquad
|\nabla^2d|\leq C,
\qquad
|a_F|\leq Cd.
\]
Consequently,
\begin{align*}
|\nabla v^F|
&\leq
C
+
\frac{C}{h}
+
\frac{C}{h}
+
C\frac{d}{h}
+
C\frac{d}{h^2}.
\end{align*}
Since \(d\leq h\) on \(\Gamma_h\) and \(h\leq1\), we conclude that
\begin{equation}\label{eq:fake-layer-gradient-pointwise}
|\nabla v^F(t,y)|
\leq
\frac{C}{h}
\qquad
\text{for a.e. }(t,y)\in(0,T)\times\Gamma_h.
\end{equation}

The gradient vanishes in the fluid region where \(d>h\). Therefore,
for \(1\leq p<\infty\) and \(\alpha\geq0\),
\begin{align*}
\|d^\alpha\nabla v^F(t)\|_{L^p(\mc{F}_0)}^p
=
\int_{\Gamma_h}
d^{\alpha p}|\nabla v^F(t,y)|^p\,dy
\leq
\frac{C}{h^p}
h^{\alpha p}|\Gamma_h|.
\end{align*}
Using \eqref{eq:fake-layer-measure}, we obtain
\begin{equation}\label{eq:fake-layer-weighted-gradient}
\|d^\alpha\nabla v^F(t)\|_{L^p(\mc{F}_0)}
\leq
C_{\alpha,p}h^{\alpha-1+1/p}.
\end{equation}
For \(p=\infty\), estimate
\eqref{eq:fake-layer-gradient-pointwise} and \(d\leq h\) give
\begin{equation}\label{eq:fake-layer-weighted-gradient-infty}
\|d^\alpha\nabla v^F(t)\|_{L^\infty(\mc{F}_0)}
\leq
C_\alpha h^{\alpha-1}.
\end{equation}

Since \(h=c\nu\), estimates
\eqref{eq:fake-layer-weighted-gradient} and
\eqref{eq:fake-layer-weighted-gradient-infty} imply for every \(1\leq p\leq\infty\) and every \(\alpha\geq0\),
$$\left\|d^\alpha\nabla v^F\right\|_
{L^\infty(0,T;L^p(\mc{F}_0))}
\leq
C_{\alpha,p}\nu^{\alpha-1+1/p}.$$
Choosing \(\alpha=0\) and \(p=2\), we obtain
$
\|\nabla v^F\|_
{L^\infty(0,T;L^2(\mc{F}_0))}
\leq
C\nu^{-1/2}.
$
Choosing \(\alpha=1\) and \(p=2\), we obtain the additional estimate
$
\|d\nabla v^F\|_
{L^\infty(0,T;L^2(\mc{F}_0))}
\leq
C\nu^{1/2}.
$
Choosing \(\alpha=1\) and \(p=\infty\), we obtain
\[
\|d\nabla v^F\|_
{L^\infty((0,T)\times\mc{F}_0)}
\leq C.
\]
Finally, choosing \(\alpha=2\) and \(p=\infty\), we obtain
$
\|d^2\nabla v^F\|_
{L^\infty((0,T)\times\mc{F}_0)}
\leq
C\nu.
$
This proves (\ref{fake-layer-gradient}).

\medskip
\noindent\textbf{Step 6: Regularity of the corrected Euler velocity.}

We identify the Euler state with
the piecewise-defined field
\[
\overline{u}^E(t,y)
:=
\begin{cases}
u^E(t,y), & y\in\mc F_0,\\[1mm]
u_{\mc S}^E(t,y), & y\in\mc S_0,
\end{cases}
\qquad
u_{\mc S}^E(t,y)
=
\ell^E(t)+\omega^E(t)\times y.
\]
Both $u^E$ and $u_{\mc S}^E$ are divergence-free in their respective
domains. Moreover, the impermeability condition
ensures that no measure concentrated on $\partial\mc S_0$ appears in
the distributional divergence of $\overline{u}^E$. Consequently,
$\operatorname{div}\overline{u}^E=0
\text{ in }\mathcal D'(\mathbb R^3).$
Since $\overline{u}^E=u_{\mc S}^E$ in $\mc S_0$, we also have
$\mathbb D(\overline{u}^E)=0
\text{ in }\mc S_0.$
Thus, $\overline{u}^E(t)\in\mc H^E$ for every $t\in[0,T]$. In what
follows, as usual, we continue to denote this global extension simply
by $u^E$.

Define the corrected Euler velocity by $\widetilde{u}_\nu^E:=u^E-v^F
\text{ in }\mathbb R^3.$
We first verify its piecewise $H^1$-regularity. For each fixed
$\nu>0$, the estimates established above imply $v^F|_{\mc F_0}
\in
L^\infty(0,T;H^1(\mc F_0)).$
 Since the Euler solution satisfies $u^E|_{\mc F_0}
\in
L^\infty(0,T;H^s(\mc F_0)),
\quad s>\frac52,$ it follows that $\widetilde{u}_\nu^E|_{\mc F_0}
\in
L^\infty(0,T;H^1(\mc F_0)).$
On the other hand, $v^F=0$ in $\mc S_0$, and hence $\widetilde{u}_\nu^E
=
u_{\mc S}^E
\text{ in }\mc S_0.$ Since $u_{\mc S}^E$ is a rigid velocity and $\mc S_0$ is bounded, we
also have
\[
\widetilde{u}_\nu^E|_{\mc S_0}
\in
L^\infty(0,T;H^1(\mc S_0)).
\]
By
\eqref{eq:fake-layer-matching},
$v^F=u^E-u_{\mc S}^E
\text{ in }\Gamma_{c\nu/2}$.
Therefore, in a neighbourhood of $\partial\mc S_0$ on the fluid side,
\[
\widetilde{u}_\nu^E
=
u^E-v^F
=
u^E-\big(u^E-u_{\mc S}^E\big)
=
u_{\mc S}^E.
\]
Consequently, $\left.\widetilde{u}_\nu^E\right|_{\partial\mc S_0,\mathrm{fluid}}
=
u_{\mc S}^E.$ Since $\widetilde{u}_\nu^E=u_{\mc S}^E$ throughout $\mc S_0$, its trace
from the solid side satisfies $\left.\widetilde{u}_\nu^E\right|_{\partial\mc S_0,\mathrm{solid}}
=
u_{\mc S}^E$.
Thus, we have $\widetilde{u}_\nu^E
\in
L^\infty(0,T;H^1(\mathbb R^3))$. Furthermore, both $u^E$ and $v^F$ are divergence-free in
$\mathcal D'(\mathbb R^3)$. Hence, $\operatorname{div}\widetilde{u}_\nu^E
=
\operatorname{div}u^E-\operatorname{div}v^F
=
0
\qquad\text{in }\mathcal D'(\mathbb R^3)$.
Since $\widetilde{u}_\nu^E=u_{\mc S}^E$ in $\mc S_0$, we also have
$\mathbb D(\widetilde{u}_\nu^E)=0
\text{ in }\mc S_0$. It follows that $\widetilde{u}_\nu^E
\in
L^\infty(0,T;\mc V)
\subset
L^2(0,T;\mc V)$. Finally, the regularity of the Euler solution gives $u^E\in C([0,T];\mc H^E)$,
where $u^E$ is understood as the global extension introduced above.
Together with \eqref{eq:fake-layer-C1H}, which gives $v^F\in C^1([0,T];\mc H^E)$,
we obtain
\[
\widetilde{u}_\nu^E
=
u^E-v^F
\in
C([0,T];\mc H^E)
\cap
L^2(0,T;\mc V).
\]
This proves \eqref{fake-layer-corrected-Euler} and completes the proof.
\end{proof}
}

Here and in the sequel, we make use of the Landau notations $o(1)$ and $O(1)$ for quantities respectively converging to $0$ and bounded with respect to the limit $\veps, \,\nu \to 0$.

\subsection{Proof of Theorem \ref{main-thm}}
This section is devoted to prove the first main result of our paper. 

\begin{proof}
Let us recall the relative energy inequality \eqref{Rel-energy} given by Theorem \ref{rel-energy}.
Regarding suitable test function in the inequality \eqref{Rel-energy}, we can take $U=u^{E}-v_{\mc{F}}$, $r=1$, and  in particular they satisfy the required property \eqref{rel:test}. 
 With these test functions, from \eqref{Expression-Relative-energy}, we write 
\begin{align}\label{Expres-E-1,U}
\mathcal E([\rho,u]  \mid [1,U])(t) = \frac{1}{2}\int\limits_{\mathbb R^3} \rho|u-U|^2 (t) \, dy + \frac{1}{\veps^2} \int\limits_{\mc F_0} \Big(  H(\rho) - \frac{\partial H}{\partial \rho}(1)(\rho-1) - H(1) \Big) (t) dy , 
\end{align}
and the remainder term \eqref{Remainder-term} of the relative energy inequality  has the following expression: 
\begin{multline}\label{Remainder-term-Sec-4.3}
\mathcal R(\rho, u, U)(t) =  
\nu  \int\limits_{\mc{F}_0}   \sigma(u) :  \mathbb D(U) \, dy  +  \int\limits_{\mc F_0}\left[ \rho \frac{\partial U}{\partial s}\cdot(U-u) + \rho(u-u_{\mc{S}}) \cdot \mathbb D(U)  \cdot (U-u) \right]  
\, dy  \\
 + m (\ell_{U}-\ell)\cdot \ell'_U  +  J_0 (\omega_U -\omega)\cdot \omega'_U 
 + \int\limits_{\mc{F}_{0}} \operatorname{det} (\rho \omega, u, U)\,dy - 
\operatorname{det}(m \ell, \omega,\ell_{U}) 
 -  \operatorname{det}(J_{0}\omega,\omega,\omega_{U}) .
\end{multline}
Our main goal is to show that  
 \begin{align}\label{goal-intermidiate}
|\mc R(\rho, u, u^E-v^F)(t)| \lesssim \mc{E}([\rho, u] \mid [1, u^E-v^F])(t) + o(1) , 
 \end{align}
 where  $o(1)$ is with respect to $L^1(0,T)$. To prove \eqref{goal-intermidiate}, we begin with the following: 
\begin{align*}
 \bigg| \nu \int\limits_{\mc F_0} \sigma(u) : \mathbb D(u^E-v^F) \, dy  \bigg|  
&\leq  \nu^{1/2} \|\nabla u^E\|_{L^2(\mc F_0)} \left(\nu^{1/2}\|\sigma (u)\|_{L^2(\mc F_0)}\right) + \nu \|\nabla v^F\|_{L^2(\Gamma_{c\nu})} \|\sigma (u)\|_{L^2(\Gamma_{c\nu})} \\
&\leq   \nu^{1/2}\|\nabla u^E\|_{L^2(\mc F_0)} \left(\nu^{1/2}\|\sigma (u)\|_{L^2(\mc F_0)}\right) + C \nu^{1/2}  \|\sigma (u)\|_{L^2(\Gamma_{c\nu})} ,
\end{align*} 
where we have used  Proposition \ref{Proposition-Fake-Layer}--point \eqref{fake-layer-gradient}. 
Then, by the uniform bound \eqref{Uniform-bound-compressible}  and the assumption \eqref{Thm-assum-2}, one has 
\begin{align}\label{4-limit-1}
\nu \int\limits_0^t\int\limits_{\mc F_0} \sigma(u): \mathbb D(U) \, dy \, ds \longrightarrow 0 \, \text{ as } \, \nu, \veps \to 0.  
\end{align}
Let us now  compute  the following:
\begin{multline}\label{Term-2-Remainder}
  \int\limits_{\mc F_0}\left[ \rho \frac{\partial U}{\partial s}\cdot(U-u) + \rho(u-u_{\mc{S}}) \cdot \nabla U  \cdot (U-u) \right]  \, dy  
+  m (\ell_{U}-\ell)\cdot \ell'_U  +  J_0 (\omega_U -\omega)\cdot \omega'_U 
  + \int\limits_{\mc{F}_{0}} \rho \operatorname{det} (\omega, u, U)\,dy \\- 
m \operatorname{det}( \ell, \omega,\ell_{U}) 
 -  \operatorname{det}(J_{0}\omega,\omega,\omega_{U})
  = \Bigg[ \int\limits_{\mc F_0} \rho \left[ \frac{\partial u^E}{\partial s} +  (u^E-u^E_{\mc S}) \cdot \nabla u^E \right] \cdot (U-u) \, dy + \int\limits_{\mc F_0} \rho \det(\omega^E, u^E, U-u) \, dy \Bigg]  \\
  + \bigg[m (\ell^E)' \cdot (\ell_U - \ell) 
   + J_0 (\omega^E)' \cdot (\omega_U - \omega)  - m \det (\ell^E, \omega^E, \ell_U-\ell) - \det(J_0 \omega^E, \omega^E, \omega_U -\omega) \bigg]  \\
   + \int\limits_{\mc F_0} \rho \left[ (u-u^E) - (u_{\mc S}-u^E_{\mc S})\right] \cdot \nabla u^E \cdot (U-u) \, dy   + \Bigg[\int\limits_{\mc{F}_{0}} \rho \operatorname{det}  (\omega, u, U)\,dy - 
m \operatorname{det}( \ell, \omega,\ell_{U})
 -  \operatorname{det}(J_{0}\omega,\omega,\omega_{U})  \Bigg] \\ - \Bigg[\int\limits_{\mc{F}_{0}} \rho \operatorname{det}  (\omega^E, u^E, U-u)\,dy -
m \operatorname{det} (\ell^E,\omega^E,\ell_U-\ell) 
 -  \operatorname{det}(J_{0}\omega^E,\omega^E,\omega_U-\omega) \Bigg]  
    - \int\limits_{\mc F_0} \rho \frac{\partial v^F}{\partial s} \cdot (U-u) \, dy \\
- m (\ell_{v^F})' \cdot (\ell_U - \ell) - J_0 (\omega_{v^F})' \cdot  (\omega_U - \omega) - \int\limits_{\mc F_0} \rho (u-u_{\mc S}) \cdot \nabla v^F \cdot (U-u) \, dy 
 : = \sum_{j=1}^9 \mc R_{j}(t) . 
\end{multline}
Below, we shall compute all the terms in the right-hand side of \eqref{Term-2-Remainder}.

\vspace{.2cm}
\noindent 
(i) We first observe that (using the equation   \eqref{chg of var momentum:fluidE} of the target system), 
\begin{multline}\label{Term-R-1}
 \int\limits_0^t \mc R_1(s) \, ds =   \int\limits_0^t\int\limits_{\mc F_0} \rho\left[ \frac{\partial u^E}{\partial s} +   (u^E-u^E_{\mc S}) \cdot \nabla u^E  \right] \cdot (U-u) \, dy \, ds 
  +  \int\limits_0^t \int\limits_{\mc F_0} \rho \det(\omega^E, u^E, U-u) \, dy \, ds \\
   = -\int\limits_0^t \int\limits_{\mc{F}_0}  \rho(U-u)\cdot \nabla p^E \, dy\, ds 
  = 
  \int\limits_0^t \int\limits_{\mc{F}_0} \rho(u-u_{\mc S}) \cdot \nabla p^E \, dy \, ds - 
  \int\limits_0^t \int\limits_{\mc{F}_0} \rho U \cdot \nabla p^E \, dy \, ds + \int\limits_0^t \int\limits_{\mc{F}_0}  \rho u_{\mc S} \cdot \nabla p^E \, dy\, ds \\
  =  \int\limits_0^t \int\limits_{\mc{F}_0} 
  \rho(u-u_{\mc S}) \cdot \nabla p^E \, dy\, ds -  \int\limits_0^t \int\limits_{\mc{F}_0} (\rho-1) U \cdot \nabla p^E \, dy\, ds - \int\limits_0^t \int\limits_{\mc{F}_0} U\cdot \nabla p^E \, dy\, ds \\
    +  \int\limits_0^t \int\limits_{\mc{F}_0} (\rho-1) u_{\mc S} \cdot \nabla p^E \,dy\, ds + \int\limits_0^t \int\limits_{\mc{F}_0}  u_{\mc S} \cdot \nabla p^E \,dy\,ds \\
  =  \int\limits_0^t \int\limits_{\mc{F}_0} 
  \rho(u-u_{\mc S})\cdot \nabla p^E \, dy\, ds 
  + \int\limits_0^t \int\limits_{\mc{F}_0} (\rho-1) u_{\mc S} \cdot \nabla p^E \, dy\, ds  -  \int\limits_0^t \int\limits_{\mc{F}_0} (\rho-1) U \cdot \nabla p^E \, dy\, ds \\+ \int\limits_0^t\int\limits_{\mc F_0} v^F \cdot \nabla p^E \, dy\, ds 
  \underbrace{- \int\limits_0^t \int\limits_{\partial \mc{S}_0} (p^E n) \cdot u^E \, d\Gamma \, ds 
  + \int\limits_0^t \int\limits_{\partial \mc{S}_0} (p^E n) \cdot u_{\mc S} \, d \Gamma \, ds}_{\mc R_{1,b}}.
\end{multline}

{
Let us focus on the first term in the right-hand side of \eqref{Term-R-1}. By means of the estimates $\eqref{Uniform-bound-compressible}_1$ and $\eqref{Uniform-bound-compressible}_3$, we have 
\begin{align}\label{Limit-rho-u}
\rho u \to \overline{\rho  u} \text{ weakly-}(*) \text{ in } L^\infty(0,T; L^{\frac{2\gamma}{\gamma+1}}(\mc F_0; \mathbb R^3)) \,  \text{ as } \veps, \nu \to 0. 
\end{align}
Moreover, for each compact set $K\subset \mc F_0$, we have 
\begin{align*}
\esssup_{t\in [0,T]}\int\limits_{K} |\rho u_{\mc S}|^{\frac{2\gamma}{\gamma+1}} \, dy < +\infty 
\end{align*}
Thus, for each compact subset $K$ of $\mc F_0$, one has  
\begin{align}\label{Limit-rho-u-s}
\rho u_{\S} \to \overline{\rho  u_\S}  \, \text{ weakly-}(*) \text{ in } L^\infty(0,T; L^{\frac{2\gamma}{\gamma+1}}(K; \mathbb R^3)) \,  \text{ as } \veps, \nu \to 0.
\end{align}
Now, we use the continuity equation \eqref{weakform-continuity}. Considering any test function $\psi\in C_c^\infty([0,T)\times \mc F_0)$  in \eqref{weakform-continuity}, we get
\begin{align}\label{R-1-1-continuity}
\int\limits_0^t \int\limits_{\mc F_0} \rho(u-u_{\S})\cdot \nabla \psi \, dy\, ds  = \bigg[\int\limits_{\mc F_0} \rho  \psi \, dy\bigg]_0^t - \int\limits_0^t \int\limits_{\mc F_0} \rho \frac{\partial \psi}{\partial s} \, dy\, ds    =
\bigg[\int\limits_{\mc F_0} (\rho-1)  \psi \, dy\bigg]_0^t - \int\limits_0^t \int\limits_{\mc F_0} (\rho-1) \frac{\partial \psi}{\partial s} \, dy\, ds .
\end{align}
But, we compute (by using the estimates \eqref{Uniform-bound-compressible})
\begin{multline*}
\Bigg| \bigg[\int\limits_{\mc F_0} (\rho-1)  \psi \, dy\bigg]_0^t - \int\limits_0^t \int\limits_{\mc F_0} (\rho-1) \frac{\partial \psi}{\partial s} \, dy\, ds   \Bigg|  \\
 \leq 
C    \left( \veps \left\|  \bigg[\frac{\rho-1}{\veps}\bigg]_{\ess}\right\|_{L^\infty(0,T;L^2(\mc F_0))} +   \| [\rho]_{\res} \|_{L^\infty(0,T;L^\gamma(\mc F_0))}  + \|1_{\res}\|_{L^\infty(0,T;L^1(\mc F_0))}   \right) \leq C (\veps+\veps^2)    ,
\end{multline*}
which converges to zero as $\veps\to 0$. 
Using this in \eqref{R-1-1-continuity} and thanks  to the existence of limits \eqref{Limit-rho-u}--\eqref{Limit-rho-u-s}, we have 
\begin{align}\label{Div-free}
\operatorname{div} \left( \overline{\rho u} - \overline{\rho u_{\S} }  \right) = 0 \,  \text{ in } \, \mathcal D^\prime(\mathbb R^3) .
\end{align}
 In particular,   
 \begin{align*}
\int\limits_0^t \int\limits_{\mc F_0} \rho (u - u_{\S}) \cdot \nabla p^E \, dy\, ds \to 0 \, \text{ as } \veps, \nu \to 0 \, \text{ in } L^q(0,T) ,  \ 1\leq q< \infty,
 \end{align*}
 thanks to the fact  \eqref{Div-free} and the boundary condition $u=u_{\S}$ on $\partial \S_0$. 
 
 Let us compute the second  term in the right-hand side of \eqref{Term-R-1}.
To do that, we first consider the following smooth function for $\veps>0$, 
\begin{align}\label{function-chi-R}
\chi_\veps(y) = \begin{dcases}  y,   \ \  \text{if } y\in B_{\veps}:=B\bigg(0,\frac{1}{\veps^\alpha}\bigg) \cap \mc F_0, \\
 \frac{y}{{\veps^\alpha }|y|}  , \ \ \text{if } y \in \mc F_0 \setminus B_{\veps},  
\end{dcases}
\end{align}
and  choose the function 
\begin{align}\label{u_S-approx}
u_{\mc S, \veps}(t, y) = \ell(t) + \omega(t) \times \chi_\veps(y)  . 
\end{align}
Observe that, 
\begin{align*}
|u_{\S,\veps}|\leq C(1+ 1/\veps^\alpha) \text{ in } (0,T)\times \mc F_0 \ \text{ and } \ u_{\mc S, \veps} \to u_{\mc S} \ \text{ as $\veps\to 0$} .  
\end{align*} 
Then, we estimate the term 
\begin{equation*}
    \widetilde {\mc R}^\veps_{1} : = \int\limits_0^t \int\limits_{\mc F_0} (\rho-1) u_{\mc S, \veps} \cdot \nabla p^E \, dy\, ds,
\end{equation*} 
%In fact, it  is clear that 
%\begin{align*}
 %\mc R^\veps_{1,3} \to \mc R_{1,3} \, \text{ as } \, R \to \infty .
%\end{align*}
 and pass to the limit w.r.t. $\veps \to 0$.  In fact, we  have
\begin{align*}
&| \widetilde {\mc R}^\veps_{1} | \leq C \left(1+ \frac{1}{\veps^\alpha}  \right) \bigg|  \int\limits_0^t\int\limits_{\mc F_0} |\rho-1| |\nabla p^E| \, dy \, ds \bigg|  \\ 
& \leq \frac{C}{\veps^\alpha} \int\limits_0^t \Bigg( \veps \left\| \bigg[\frac{\rho-1}{\veps}\bigg]_{\ess}\right\|_{L^2(\mc F_0)} \|\nabla p^E \|_{L^2(\mc F_0)}  + \| [\rho]_{\res} \|_{L^\gamma(\mc F_0)} \|\nabla p^E\|_{L^\frac{\gamma}{\gamma-1}(\mc F_0)} \\ 
& \qquad \qquad \qquad + \|1_{\res}\|_{L^1(\mc F_0)} \|\nabla p^E\|_{L^\infty(\mc F_0)}  \Bigg)\, ds \\
& \leq \frac{C}{\veps^\alpha}  \Bigg( \veps \left\| \bigg[\frac{\rho-1}{\veps}\bigg]_{\ess}\right\|_{L^\infty(0,T;L^2(\mc F_0))} \|\nabla p^E \|_{L^1(0,T;L^2(\mc F_0))}  + \| [\rho]_{\res} \|_{L^\infty(0,T;L^\gamma(\mc F_0))} \|\nabla p^E\|_{L^1(0,T;L^\frac{\gamma}{\gamma-1}(\mc F_0))} \\ 
& \qquad \qquad \qquad + \|1_{\res}\|_{L^\infty(0,T;L^1(\mc F_0))} \|\nabla p^E\|_{L^1(0,T;L^\infty(\mc F_0))}  \Bigg) \\
& \leq \frac{C}{\veps^{\alpha}} (\veps + \veps^{2/\gamma} +\veps^2 ) ,
\end{align*}
where we have used the estimates \eqref{Uniform-bound-compressible} and the fact that $\nabla p^E \in C([0,T]; H^{s-1}(\mc F_0))$ with $s>\frac{5}{2}$ (see Theorem \ref{Thm-Euler-body}).
 Thus, $|\widetilde{\mc R}^\veps_{1}|\to 0$ as $\veps \to 0$ provided $\alpha < \min\{1,  \frac{2}{\gamma}\}$.

 }

%In the above construction,   one may particularly choose $R=\frac{1}{\sqrt{\veps}}$ so that 
%\begin{align}
%\lim_{\veps \to 0} \lim_{R\to \infty} \mc R^R_{1,4} \leq \lim_{\veps \to 0} C \sqrt{\veps} =0. 
%\end{align}

The third term in the right-hand side of \eqref{Term-R-1} can also be bounded by $C\veps$. Indeed, we estimate as follows (recall that $U=u^E-v^F$):
%\begin{align*}
%\mc R^\veps_{1,3}: = \int\limits_0^t \int\limits_{\mc F_0} (\rho-1) U \cdot \nabla p^E \, dy\, ds 
%\end{align*}
\begin{align}\label{term-rho-1-U-p-E}
&\bigg| \int\limits_0^t \int\limits_{\mc F_0} (\rho-1) U\cdot \nabla p^E \, dy\, ds \bigg|  \notag \\ 
%& \leq 
%\bigg| \int\limits_0^t \int\limits_{\mc F_0} (\rho-1) u^E\cdot \nabla p^E \, dy\, ds \bigg|+ \bigg| \int\limits_0^t \int\limits_{\Gamma_{c\nu}} (\rho-1) v^F\cdot \nabla p^E \, dy\, ds \bigg| \\
&\leq \|U\|_{L^\infty(0,T)\times \mc F_0)} \Bigg( \veps \left\| \bigg[\frac{\rho-1}{\veps}\bigg]_{\ess}\right\|_{L^\infty(0,T;L^2(\mc F_0))} \|\nabla p^E \|_{L^1(0,T;L^2(\mc F_0))}  \notag \\
& \qquad \qquad 
+ \| [\rho]_{\res} \|_{L^\infty(0,T;L^\gamma(\mc F_0))} \|\nabla p^E\|_{L^1(0,T;L^\frac{\gamma}{\gamma-1}(\mc F_0))} 
 + \|1_{\res}\|_{L^\infty(0,T;L^1(\mc F_0))} \|\nabla p^E\|_{L^1(0,T;L^\infty(\mc F_0))}  \Bigg) \notag \\
& \leq C (\veps + \veps^{2/\gamma} + \veps^2) ,
\end{align}
using   Theorem \ref{Thm-Euler-body} and the estimates \eqref{Uniform-bound-compressible}. 
Therefore, the third term in the right-hand side of 
\eqref{Term-R-1} vanishes as $\veps \to 0$.

%provided $\alpha < \min\{\frac{1}{3}, \frac{2}{3\gamma}\}. $

On the other hand, by using Proposition \ref{Proposition-Fake-Layer}--point \eqref{fake-layer-zero-order}, the term $\int\limits_0^t \int\limits_{\mc F_0} v^F \cdot \nabla p^E \, dy\, ds$ can be made small and is of order $\nu^{1/2}$, and thus,
\begin{align}
\bigg|  \int\limits_0^t \int\limits_{\mc F_0} v^F \cdot \nabla p^E \, dy\, ds \bigg| \to 0 \, \text{ as } \, \nu \to 0. 
\end{align}

%{\color{magenta} By using the continuity equation for the compressible fluid, we have 
%\begin{align*}
%&\int\limits_0^t \int\limits_{\mc F_0} p^E \operatorname{div} (\rho (u-u_{\mc S}))\, dy \, ds=
%\int\limits_0^t \int\limits_{\mc F_0} p^E  \partial_t \rho \, dy\, ds \\
%& = - \int\limits_0^t \int\limits_{\mc F_0}  \partial_t p^E (\rho -1) + \int\limits_{\mc F_0} p^E (\rho-1) \bigg|_{s=0}^t\\
%& = - \veps \int\limits_0^t \int\limits_{\mc F_0}  \partial_t p^E \frac{(\rho -1)}{\veps} + \veps\int\limits_{\mc F_0} p^E \frac{(\rho-1)}{\veps} \bigg|_{s=0}^t . 
%\end{align*}
%In fact, we have 
%$$ \left\| \frac{\rho-1}{\veps} \right\|_{L^2(\mc F_0)} \leq C$$ which is enough to deduce  the terms in the above relation.
%}

\vspace*{.2cm}
\noindent 
(ii) Let us focus on the term $\mc R_2$. Using \eqref{chg of var linear momentum:bodyE}--\eqref{chg of var angular momentum:bodyE}, we have 
\begin{align}\label{Term-R-2}
&\int\limits_0^t \mc R_2(s) \, ds \notag \\
 =& \int\limits_0^t \bigg[m (\ell^E)' \cdot (\ell_U - \ell) 
   + J_0 (\omega^E)' \cdot (\omega_U - \omega)  - m \det (\ell^E, \omega^E, \ell_U-\ell) - \det(J_0 \omega^E, \omega^E, \omega_U -\omega) \bigg] \, ds \notag \\
  =& \int\limits_0^t \int\limits_{\partial \mc S_0} p^E n \cdot (\ell^E-\ell) \, d\Gamma \, ds  + \int\limits_0^t \int\limits_{\partial \mc S_0} (y \times p^E n ) \cdot (\omega^E - \omega) \, d\Gamma \, ds \notag \\
  & - \underbrace{\int\limits_0^t \bigg[m (\ell^E)' \cdot \ell_{v^F} 
   + J_0 (\omega^E)' \cdot \omega_{v^F}   - m \det (\ell^E, \omega^E,  \ell_{v^F}) - \det(J_0 \omega^E, \omega^E, \omega_{v^F} ) \bigg] \, ds}_{\mc R_{2,v^F}} .
\end{align}
Now, observe that if we add  \eqref{Term-R-1} and \eqref{Term-R-2}, the first two terms in the right-hand side of \eqref{Term-R-2} is canceled against the term $\mc R^{1,b}$ of \eqref{Term-R-1} (since $u^E=u^E_{\mc S}=\ell^E + \omega^E\times y$ and $u_{\mc S}=\ell + \omega\times y$ on $\mc S_0$).

Thanks to Proposition \ref{Proposition-Fake-Layer}, we further deduce that the term $\mc R_{2,v^F}$ is small. Indeed, one can obtain 
\begin{align}\label{sub-esti-Term-R-2}
|\mc R_{2,v^F}| \leq C \|v^F\|_{C([0,T];\mc H^E)}    \leq C \nu^{1/2}. 
\end{align}

Therefore, by means of above computations, we have 
\begin{align}\label{Add-T=1=2}
\bigg|\int\limits_0^t \mc R_1(s) \, ds + \int\limits_0^t \mc R_2(s) \, ds    \bigg| \to 0 \, \text{ as } \, \veps, \nu \to 0. 
\end{align}

\vspace{.2cm}
\noindent 
(iii)
Let us consider the term $\mc R_3$, we get (recall that $U=u^E-v^F$)
\begin{align}\label{Sub-Esti-2}
|\mc R_3(t)| &= \Bigg| \int\limits_{\mc F_0} \rho\left[(u-u^E) - (u_{\mc S}-u^E_{\mc S}) \right]\cdot \nabla u^E \cdot (U-u) \, dy   \Bigg| \notag \\
& \leq C\|\nabla u^E\|_{L^\infty((0,T)\times\mc F_0)} \bigg[  \int\limits_{\mc{F}_0}\rho |u-u^E|^2 \, dy  + \int\limits_{\Gamma_{c
\nu}} \rho |v^F|^2 \, dy \bigg] 
\notag \\
& \ + {C \bigg[\int\limits_{\mc F_0}\rho|u-u^E|^2 \, dy   + 
\int\limits_{\mc F_0}\rho|u_{\mc S} -u^E_{\mc S}|^2 |\nabla u^E|^2 \, dy +\int\limits_{\Gamma_{c\nu}} \rho|v^F|^2 \, dy \bigg] }.
\end{align}
To this end, let us compute the following integrals. 
First, we find 
\begin{align}\label{esti-rho-v^F}
 \int\limits_{\Gamma_{c\nu}} \rho |v^F|^2 \, dy 
& \leq C\|v^F\|^2_{L^\infty((0,T)\times \mathbb R^3)} \left(\|[\rho]_{\res}\|_{L^\gamma(\mc F_0)} + \|[1]_{\res}\|_{L^1(\mc F_0)}\right) \notag \\
& \quad +  C\|v^F\|^2_{L^\infty((0,T)\times \mathbb R^3)} \veps \left\| \bigg[\frac{\rho-1}{\veps}\bigg]_{\ess}\right\|_{L^2(\mc F_0)}  + \int\limits_{\Gamma_{c\nu}} |v^F|^2 \, dy \notag  \\
& \leq C (\veps^{2/\gamma} + \veps^2 +   \veps +  \nu) ,
\end{align}
thanks to the estimates \eqref{Uniform-bound-compressible} and Proposition \ref{Proposition-Fake-Layer}--point \eqref{fake-layer-zero-order}.  This yields that the above integral converges to $0$ as $\veps, \nu \to 0$. 

{
Secondly, we calculate the following:
\begin{align}\label{R-3-term-u-s}
\int\limits_{\mc F_0} \rho |u_{\mc S}-u^E_{\mc S}|^2 |\nabla u^E|^2 \, dy = \int\limits_{\mc F_0} (\rho-1) |u_{\mc S}-u^E_{\mc S}|^2 |\nabla u^E|^2 \, dy +  \int\limits_{\mc F_0} |u_{\mc S}-u^E_{\mc S}|^2 |\nabla u^E|^2 \, dy.
\end{align}
Let us focus on the first term in the  right-hand side of \eqref{R-3-term-u-s}. We estimate this term  with the  truncated velocity $u_{\mc S,\veps}$ for the rigid body  (as defined by \eqref{function-chi-R}--\eqref{u_S-approx}), and pass to the limit w.r.t. $\veps.$ Indeed, we have 
\begin{align}\label{R-3-tilde-eps}
&{\widetilde{\mc R}^{\veps}_3:=\int\limits_0^t \int\limits_{\mc F_0} (\rho -1) |u_{\S,\veps}-u^E_{\S,\veps}|^2 |\nabla u^E|^2 \, dy\, ds }  \notag \\
&\leq \frac{C}{\veps^{2\alpha}}  \Bigg( \veps \left\| \bigg[\frac{\rho-1}{\veps}\bigg]_{\ess}\right\|_{L^\infty(0,T;L^2(\mc F_0))} \|\nabla u^E\|_{L^2(0,T; L^4(\mc F_0))} + \| [\rho]_{\res} \|_{L^\infty(0,T;L^\gamma(\mc F_0))} 
\|\nabla u^E\|_{L^2(0,T; L^{\frac{2\gamma}{\gamma-1}}(\mc F_0))}
\notag \\
& \qquad \qquad + \|1_{\res}\|_{L^\infty(0,T;L^1(\mc F_0))}  \|\nabla u^E\|_{L^2(0,T; L^\infty(\mc F_0))}  \Bigg) \notag \\
& \leq \frac{C}{\veps^{2\alpha}} (\veps + \veps^{2/\gamma}+\veps^2),
\end{align}
where we have used the estimates \eqref{Uniform-bound-compressible} and the fact that $u^E \in C([0,T]; H^{s}(\mc F_0))$ for  $s>\frac{5}{2}$ (see Theorem \ref{Thm-Euler-body}).
 Thus, $|\widetilde{\mc R}^\veps_{3}|\to 0$ as $\veps \to 0$ provided $\alpha < \min\{\frac12,  \frac{1}{\gamma}\}$. 
 
 Next, the second term in the right-hand side of \eqref{R-3-term-u-s} can be computed as follows:
\begin{multline}\label{R-3-u-1}
\int\limits_0^t\int\limits_{\mc F_0} |u_{\mc S}-u^E_{\mc S}|^2 |\nabla u^E|^2 \, dy  
\leq  \int\limits_0^t \int\limits_{\mc F_0} \left(|\ell-\ell^E|^2 +  |\omega-\omega^E|^2 |y|^2\right)|\nabla u^E|^2 \, dy\, ds \\
  \leq \int\limits_0^t|\ell-\ell^E|^2 \int\limits_{\mc F_0}|\nabla u^E|^2 \, dy \, ds  + \int\limits_0^t |\omega-\omega^E|^2 \int\limits_{\mc F_0} |y|^2 |\nabla u^E|^2 \, dy \, ds  \\
   \leq \int\limits_0^t|\ell-\ell^E|^2 \|\nabla u^E\|^2_{L^2(\mc F_0)} \, ds +    \int\limits_0^t|\omega-\omega^E|^2 \|\nabla u^E\|^2_{L^2(\mc F_0, (1+|y|^2)^{\frac12} dy)} ds  \\ 
  \leq \|\nabla u^E\|^2_{L^\infty(0,T;L^2(\mc F_0))} \int\limits_0^t |\ell-\ell^E|^2 \, ds + 
\|\nabla u^E\|^2_{L^\infty(0,T;L^2(\mc F_0, (1+|y|^2)^{\frac12} dy))} \int\limits_0^t |\omega-\omega^E|^2 \, ds  
\leq C \int\limits _0^t \| u-u^E\|^2_{\mc H} \, ds 
\end{multline} 
(recall the definition of $\mc H$-norm associated to the inner product \eqref{inner product:W}).}

Thus, gathering the estimates \eqref{Sub-Esti-2}--\eqref{R-3-u-1},  we have
\begin{align*}
\bigg| \int\limits_0^t \mc R_3(s)\, ds \bigg| \leq C  \int\limits_0^t \|u-u^E\|^2_{\mc H} \, ds + o(1). 
\end{align*}

%using   Proposition \ref{Proposition-Fake-Layer}--point (2). 
%Therefore, 
%\begin{align}\label{Final-R-3}
%\int\limits_0^t \mc R_3(s) \, ds \leq C \int\limits_0^t \|(u-u^E)(s)|^2_{\mc H} \, ds + C  \nu^{\frac{\gamma-1}{\gamma}} . 
%\end{align}

\vspace{.2cm}
\noindent
(iv) We now consider the terms $\mc R_4$ and $\mc R_5$; summing up we have
\begin{align}\label{Term-R-3-R-4}
 \mc R_4 (t) + \mc R_5(t)
& = \Bigg[\int\limits_{\mc{F}_{0}} \rho \operatorname{det}  (\omega, u, U)\,dy - 
m \operatorname{det}( \ell, \omega,\ell_{U})  
 -  \operatorname{det}(J_{0}\omega,\omega,\omega_{U})  \Bigg]  \notag \\
&   - \Bigg[\int\limits_{\mc{F}_{0}} \rho \operatorname{det}  (\omega^E, u^E, U-u)\,dy -
m \operatorname{det} (\ell^E,\omega^E,\ell_U-\ell) 
 -  \operatorname{det}(J_{0}\omega^E,\omega^E,\omega_U-\omega) \Bigg]  \notag \\
& =  \Bigg[\int\limits_{\mc{F}_{0}} \rho \operatorname{det}  (\omega, u, u^E-v^F) \,dy - 
m \operatorname{det}( \ell,\omega,\ell^{E}-\ell_{v^F}) 
 -  \operatorname{det}(J_{0}\omega,\omega,\omega^E -\omega_{v^F}) \Bigg]  \notag  \\
 &  +  \Bigg[ \int\limits_{\mc{F}_{0}} \rho \operatorname{det}  (\omega^E, u^E, u)\,dy - 
m \operatorname{det}(\ell^E, \omega^E, \ell) 
 -  \operatorname{det}(J_{0}\omega^E,\omega^E,\omega)\Bigg]   
 \notag \\ 
 & +  \Bigg[ \int\limits_{\mc{F}_{0}} \rho \operatorname{det}  (\omega^E, u^E, v^F)\,dy - 
m \operatorname{det}( \ell^E, \omega^E,\ell_{v^F}) 
 -  \operatorname{det}(J_{0}\omega^E,\omega^E,\omega_{v^F})\Bigg]  \notag \\
&  = \Bigg[\int\limits_{\mc{F}_{0}} \rho \operatorname{det}  (\omega, u, u^E)\,dy +  
m \operatorname{det}(\omega,\ell, \ell^E) 
 -  \operatorname{det}(J_{0}\omega,\omega,\omega^E) \Bigg]  \notag \\ 
&  -  \Bigg[ \int\limits_{\mc{F}_{0}} \rho \operatorname{det}  (\omega^E, u, u^E)\,dy + 
m \operatorname{det}( \omega^E,\ell,\ell^E) 
 -  \operatorname{det}(J_{0}\omega^E,\omega,\omega^E)\Bigg] \notag \\ 
&  -  \Bigg[ \int\limits_{\mc{F}_{0}} \rho \operatorname{det}  (\omega, u, v^{F})\,dy - 
m \operatorname{det}(\ell,\omega,\ell_{v^F}) 
 -  \operatorname{det}(J_{0}\omega,\omega,\omega_{v^F})\Bigg] \notag \\
&    +  \Bigg[ \int\limits_{\mc{F}_{0}} \rho \operatorname{det}  (\omega^E, u^E, v^F)\,dy - 
m \operatorname{det}(\ell^E,\omega^E,\ell_{v^F}) 
 -  \operatorname{det}(J_{0}\omega^E,\omega^E,\omega_{v^F})\Bigg] .
\end{align}
Observe that the first two terms in the right-hand side of \eqref{Term-R-3-R-4} satisfy
\begin{multline}\label{term-R34-sub1}
\Bigg|\bigg(\int\limits_{\mc{F}_{0}} \rho \operatorname{det}  (\omega, u, u^E)\,dy + 
m \operatorname{det}( \omega,\ell,\ell^E) 
 -  \operatorname{det}(J_{0}\omega,\omega,\omega^E) \bigg) \\  
    - \bigg(  \int\limits_{\mc{F}_{0}} \rho \operatorname{det}  (\omega^E, u, u^E)\,dy + 
m \operatorname{det}( \omega^E,\ell,\ell^E) 
 -  \operatorname{det}(J_{0}\omega^E,\omega,\omega^E)\bigg)\Bigg|   \\
 %= \Bigg| \int\limits_{\mc F_0} \rho \operatorname{det}(\omega-\omega^E, u, u^E)\, dy + m \operatorname{det}(\omega-\omega^E,\ell, \ell^E) - \operatorname{det}(J_0\omega -J_0\omega^E, \omega, \omega^E ) \Bigg|  \\
  =\Bigg| \int\limits_{\mc F_0} \rho \operatorname{det}(\omega-\omega^E, u-u^E, u^E)\, dy + m \operatorname{det}(\omega-\omega^E,\ell-\ell^E, \ell^E) - \operatorname{det}(J_0\omega -J_0\omega^E, \omega-\omega^E, \omega^E )\Bigg|   \\
  \leq  { \int\limits_{\mc F_0} \rho |u-u^E|^2 \, dy + |\omega-\omega^E|^2 \int\limits_{\mc F_0} \rho |u^E|^2\, dy + C (|\ell-\ell^E|^2 + |\omega-\omega^E|^2) } \\
 \leq  \int\limits_{\mc F_0} \rho |u-u^E|^2 \, dy  + C (|\ell-\ell^E|^2 + |\omega-\omega^E|^2) + 
  |\omega-\omega^E|^2\int\limits_{\mc F_0} |u^E|^2 \, dy + 
   |\omega-\omega^E|^2\int\limits_{\mc F_0} (\rho-1)|u^E|^2 \, dy \\
  \leq  C\|u-u^E\|^2_{\mc H} + |\omega-\omega^E|^2\int\limits_{\mc F_0} (\rho-1)|u^E|^2 \, dy . 
\end{multline}
Let us estimate the integral $\int\limits_{\mc F_0} (\rho-1)|u^E|^2 \, dy$ in the above inequality. Indeed, we find 
\begin{align*}
&\int\limits_{\mc F_0} (\rho-1)|u^E|^2 \, dy \\
&\leq  \veps \left\| \bigg[\frac{\rho-1}{\veps}\bigg]_{\ess}\right\|_{L^2(\mc F_0))}  \| u^E \|^2_{L^4(\mc F_0)} + \| [\rho]_{\res} \|_{L^\gamma(\mc F_0))} \|u^E\|^2_{L^{\frac{2\gamma}{\gamma-1}}(\mc F_0) } + \|1_{\res}\|_{L^1(\mc F_0))}  \|u^E\|_{L^\infty(\mc F_0)}. 
%& \leq \frac{C}{\veps^{3\alpha}} (\veps + \veps^{2/\gamma}+\veps^2),
\end{align*}
Using \eqref{Uniform-bound-compressible} and the fact that $u^E \in C([0,T]; H^{s}(\mc F_0))$ for  $s>\frac{5}{2}$ (see Theorem \ref{Thm-Euler-body}), we further find that 
\begin{align}\label{term-R-45-eps}
\int\limits_0^t \int\limits_{\mc F_0} (\rho-1)|u^E|^2 \, dy \, ds \leq C (\veps + \veps^{2/\gamma}+\veps^2) , 
\end{align}
which eventually converges to $0$ as $\veps\to 0$.

Next, we have 
\begin{align}\label{term-R34-sub2}
&\Bigg|  \int\limits_{\mc{F}_{0}} \rho \operatorname{det}  (\omega, u, v^{F})\,dy - 
m \operatorname{det}( \ell,\omega,\ell_{v^F}) 
 -  \operatorname{det}(J_{0}\omega,\omega,\omega_{v^F}) \Bigg| \notag \\
 &\leq C \| \sqrt{\rho}u\|_{L^2(\mc F_0)}
 \|\sqrt{\rho}v^F\|_{L^2(\mc F_0)}  + C \|v_F\|_{\mc H^E} \notag \\
 & \leq C (\veps^{2/\gamma}+\veps^2+\veps + \nu)^{1/2} + C \nu^{1/2},
\end{align}
by using the uniform bound $\eqref{Uniform-bound-compressible}_1$, the estimate \eqref{esti-rho-v^F} and Proposition \ref{Proposition-Fake-Layer}--point \eqref{fake-layer-zero-order}. 

Finally, the last couple of terms in the right-hand side of \eqref{Term-R-3-R-4} enjoy the following estimate
\begin{align}\label{term-R34-sub3}
&\Bigg|  \int\limits_{\mc{F}_{0}} \rho \operatorname{det}  (\omega^E, u^E, v^{F})\,dy - 
m \operatorname{det}( \ell^E,\omega^E,\ell_{v^F}) 
 -  \operatorname{det}(J_{0}\omega^E,\omega^E, \omega_{v^F} ) \Bigg| \notag \\
 & 
 =\Bigg| \int\limits_{\Gamma_{c\nu}} (\rho -1)\operatorname{det}  (\omega^E, u^E, v^{F})\,dy + \int\limits_{\Gamma_{c\nu}}  \operatorname{det}  (\omega^E, u^E, v^{F})\,dy - 
m \operatorname{det}( \omega^E,\ell^E,\ell_{v^F}) 
 -  \operatorname{det}(J_{0}\omega^E,\omega^E,\omega_{v^F}) \Bigg| \notag \\
 &\leq  C \left( \veps \left\| \bigg[\frac{\rho-1}{\veps}\bigg]_{\ess}\right\|_{L^2(\mc F_0)} +  \|[\rho]_{\res}\|_{L^\gamma(\mc F_0)} + \|1_{\res}\|_{L^1(\mc F_0)}  \right) +
 C  \|v^F\|_{\mc H^E} \notag \\
 & \leq C( \veps + \veps^2)  +
 C \nu^{1/2} ,
\end{align}
thanks to $\eqref{Uniform-bound-compressible}_{2,3}$, and Proposition \ref{Proposition-Fake-Layer}--point \eqref{fake-layer-zero-order}.

Therefore, combining the estimates \eqref{term-R34-sub1}--\eqref{term-R34-sub3} in \eqref{Term-R-3-R-4}, we have 
\begin{align}\label{Final-esti-R-3-4}
\bigg|\int\limits_0^t \left(\mc R_4(s)+\mc R_5(s)\right) \, ds \bigg| 
%& \leq C \int\limits_0^t \left\|(u-u^E)(s)\right\|^2_{\mc H}   \|u^E(s)\|_{\mc H^E} \, ds  + C \int\limits_0^t \|u(s)\|^2_{\mc H} \|v^F(s)\|_{\mc H^E} \, ds \notag \\
%& + C
% \int\limits_0^t \|u^E(s)\|^2_{\mc H^E} \|v^F(s)\|_{\mc H^E} \, ds + C\veps  \notag \\
 \leq C \int\limits_0^t \|(u-u^E)(s)\|^2_{\mc H} \, ds + o(1) .
\end{align}

\vspace{.2cm}
\noindent 
(v)
Let us consider the terms $\mc R_6$, $\mc R_7$ and $\mc R_8$. 
We have
\begin{align}\label{Term-R-6-8-9}
&|\mc R_6(t)+ \mc R_7(t)+ \mc R_8(t)| \notag \\
=& \bigg|  \int\limits_{\mc F_0} \rho \frac{\partial v^F}{\partial s} \cdot (u^E-v^F-u) \, dy +   m \ell_{v^F}' \cdot (\ell^E -\ell_{v^F} - \ell) + J_0 \omega_{v^F}' \cdot  (\omega^E - \omega_{v^F} - \omega) 
\bigg| \notag \\
 \leq &
\bigg| \int\limits_{\mc F_0} \rho \frac{\partial v^F}{\partial s} \cdot (u^E-u) \, dy +   m \ell_{v^F}' \cdot (\ell^E  - \ell) + J_0 \omega_{v^F}' \cdot  (\omega^E  - \omega) 
\bigg| \notag \\
  & + \bigg| \int\limits_{\mc F_0} \rho \frac{\partial v^F}{\partial s} \cdot  v^F \, dy +   m \ell_{v^F}' \cdot \ell_{v^F} + J_0 \omega_{v^F}' \cdot  \omega_{v^F}  
\bigg| \notag \\
 =& \mc J_1(t) + \mc J_2(t)  .
\end{align}
Using Cauchy-Schwarz and Young's inequality, the term $\mc J_1$ can be estimated as 
\begin{align}\label{term-J-1}
\mc J_1 (t) &  = 
\bigg| \int\limits_{\mc F_0} \sqrt{\rho}\frac{\partial v^F}{\partial s} \cdot \sqrt{\rho} (u^E-u) \, dy +   \sqrt{m} \, \ell'_{v^F} \cdot \sqrt{m}\, (\ell^E-\ell) 
+ \sqrt{|J_0|} \,  \omega'_{v^F} \cdot \sqrt{|J_0|} \, (\omega^{E}-\omega) \bigg| \notag \\
& \leq 
\frac{1}{2} \bigg( \int\limits_{\mc F_0} \rho|u-u^E|^2 \, dy + m |\ell^E-\ell|^2 + J_0|\omega^E-\omega|^2   \bigg) + 
\frac{1}{2} \bigg( \int\limits_{\Gamma_{c\nu}} \rho\bigg|\frac{\partial v^F}{\partial s}\bigg|^2 \, dy + m |\ell'_{v^F}|^2 + J_0|\omega'_{v^F}|^2   \bigg) \notag \\
& \leq \frac{1}{2} \|u-u^E\|^2_{\mathcal H} +   
\frac{1}{2} \bigg\|\frac{\partial v^F}{\partial s}\bigg\|^2_{\mc H^E} + \frac{1}{2}\int\limits_{\Gamma_{c\nu}} (\rho-1)\bigg|\frac{\partial v^F}{\partial s}\bigg|^2 \, dy \notag \\ 
& \leq 
\frac{1}{2}\|u-u^E\|^2_{\mathcal H} + C \nu  + C {\bigg\|\frac{\partial v^F}{\partial s}\bigg\|_{L^\infty(\mathbb R^3)} } \bigg( \veps \bigg\| \bigg[\frac{\rho-1}{\veps}\bigg]_{\ess} \bigg\|_{L^2(\mc F_0)}  + \|[\rho]_{\res}\|_{L^\gamma(\mc F_0)} + \|1_{\res} \|_{L^1(\mc F_0)} \bigg) \notag \\
&\leq \frac{1}{2}\|u-u^E\|^2_{\mathcal H} + o(1),
\end{align}
where we have utilized Proposition \ref{Proposition-Fake-Layer}--point \ref{fake-layer-zero-order} and the bound \eqref{Uniform-bound-compressible}. 

On the other hand, $\mc J_2$ satisfies the following:
\begin{align}\label{term-J-2}
\mc J_2(t) & \leq 
\int\limits_{\mc F_0} (\rho-1)\bigg| v^F \cdot \frac{\partial v^F}{\partial s} \bigg|  \, dy + 
\int\limits_{\mc F_0} \bigg| v^F \cdot \frac{\partial v^F}{\partial s} \bigg|  \, dy 
+
\big| m \ell_{v^F}' \cdot \ell_{v^F}\big| + \big|J_0 \omega_{v^F}' \cdot  \omega_{v^F}\big| \notag \\
& \leq C \|v^F\|_{L^\infty(\mathbb R^3)} \bigg\| \frac{\partial v^F}{\partial s}\bigg\|_{L^\infty(\mathbb R^3)} \bigg( \veps \bigg\| \bigg[\frac{\rho-1}{\veps}\bigg]_{\ess} \bigg\|_{L^2(\mc F_0)}  + \|[\rho]_{\res}\|_{L^\gamma(\mc F_0)} + \|1_{\res} \|_{L^1(\mc F_0)} \bigg) \notag \\
& \quad  + \frac{1}{2}\|v^F\|^2_{\mc H^E} + \frac{1}{2}\bigg\| \frac{\partial v^F}{\partial s} \bigg\|^2_{\mc H^E} \notag \\
& \leq C (\veps + \veps^{2/\gamma} + \veps^2 ) + C \nu , 
\end{align}
using   Proposition \ref{Proposition-Fake-Layer}--point \ref{fake-layer-zero-order} and the uniform bound \eqref{Uniform-bound-compressible}.

Thus, combining the estimates \eqref{term-J-1}--\eqref{term-J-2} in \eqref{Term-R-6-8-9}, we have
\begin{align}\label{Final-R-6-7-8}
 \int\limits_0^t \left( \mc R_6(s) + \mc R_7 (s) + \mc R_8(s) \right) ds \leq 
 \int\limits_0^t \|(u-u^E)(s)\|^2_{\mc H} \, ds + o(1).  
\end{align}

\vspace{.2cm}
\noindent 
(vi) We now focus on the term $\mc R_9$: 
\begin{align} 
\mc R_9(t) & = \int\limits_{\mc F_0}\rho(u-u_{\mc S}) \cdot \nabla v^F \cdot (U-u) \, dy  \notag \\
& =  \int\limits_{\mc F_0}\rho(u-u_{\mc S}) \cdot \nabla v^F \cdot u^E \, dy  - 
\int\limits_{\mc F_0}\rho(u-u_{\mc S}) \cdot \nabla v^F \cdot v^F \, dy  -
\int\limits_{\mc F_0}\rho(u-u_{\mc S}) \cdot \nabla v^F \cdot u \, dy \notag \\
& = : \mc R_{9,1}(t) + \mc R_{9,2}(t) + \mc R_{9,3}(t) .
\end{align}
Set $w=u-u_{\mc S}$ and define
\[
A_\nu(s)
:=
\left\|
\frac{\rho w}{d^2}
\right\|_{L^{6/5}(\Gamma_{c\nu})},
\qquad
B_\nu
:=
\nu\int\limits_0^T A_\nu(s)^2\,ds.
\]
By the boundary-layer assumption \eqref{Thm-assum-2}, we have
$B_\nu\to0$ as $\nu\rightarrow 0$.
Recall from Proposition \ref{Proposition-Fake-Layer} that the boundary-layer corrector satisfies 
\[
\operatorname{supp}v^F\subset\Gamma_{c\nu},
\qquad
\|d^2\nabla v^F\|_{L^\infty((0,T)\times\mc F_0)}
\le C\nu,
\qquad
\|v^F\|_{L^\infty((0,T)\times\mc F_0)}
\le C.
\]
Since $|\Gamma_{c\nu}|\le C\nu$ and $u^E$ is uniformly
bounded in a fixed neighbourhood of $\partial\mc S_0$, we obtain
\[
\|u^E(s)\|_{L^6(\Gamma_{c\nu})}
+
\|v^F(s)\|_{L^6(\Gamma_{c\nu})}
\le C\nu^{1/6},
\qquad s\in[0,T].
\]
Here and below, $C$ is independent of the small parameters.

Now, for the first term,
H\"older's inequality in space gives
\begin{align*}
\left| \mc R_{9,1}(s)\, ds \right|=\bigg|
\int\limits_{\mc F_0}
\rho\bigl((w\cdot\nabla)v^F\bigr)\cdot u^E\,dy
\bigg|
&\le
\int\limits_{\Gamma_{c\nu}}
\frac{\rho|w|}{d^2}\,
d^2|\nabla v^F|\,|u^E|\,dy
\\
&\le
A_\nu(s)
\|d^2\nabla v^F(s)\|_{L^\infty(\Gamma_{c\nu})}
\|u^E(s)\|_{L^6(\Gamma_{c\nu})}
\\
&\le C\nu^{7/6}A_\nu(s).
\end{align*}
Integrating in time and applying the Cauchy--Schwarz inequality,
we infer that, for every $t\in[0,T]$,
\begin{align}
\bigg|\int\limits_{0}^t \mc R_{9,1}(s)\, ds \bigg|\leq \int\limits_0^t
\bigg|
\int\limits_{\mc F_0}
\rho\bigl((w\cdot\nabla)v^F\bigr)\cdot u^E\,dy
\bigg|\,ds
&\le
C\nu^{7/6}\int\limits_0^t A_\nu(s)\,ds
\nonumber\\
&\le
C\sqrt{T}\,\nu^{7/6}
\bigg(\int\limits_0^T A_\nu(s)^2\,ds\bigg)^{1/2}
\nonumber\\
&=
C\sqrt{T}\,\nu^{2/3}B_\nu^{1/2}
\longrightarrow0.
\label{Term-R-71}
\end{align}

For the second term, the same argument yields
\begin{align*}
\left|\mc R_{9,2}(s)\, ds \right|\bigg|
\int\limits_{\mc F_0}
\rho\bigl((w\cdot\nabla)v^F\bigr)\cdot v^F\,dy
\bigg|
&\le
\int\limits_{\Gamma_{c\nu}}
\frac{\rho|w|}{d^2}\,
d^2|\nabla v^F|\,|v^F|\,dy
\\
&\le
A_\nu(s)
\|d^2\nabla v^F(s)\|_{L^\infty(\Gamma_{c\nu})}
\|v^F(s)\|_{L^6(\Gamma_{c\nu})}
\\
&\le C\nu^{7/6}A_\nu(s).
\end{align*}
Consequently,
\begin{align}
\bigg|\int\limits_{0}^t \mc R_{9,2}(s)\, ds \bigg|\leq \int\limits_0^t 
\bigg|
\int\limits_{\mc F_0}
\rho\bigl((w\cdot\nabla)v^F\bigr)\cdot v^F\,dy
\bigg|\,ds
&\le
C\nu^{7/6}\int\limits_0^t A_\nu(s)\,ds
\nonumber\\
&\le
C\sqrt{T}\,\nu^{2/3}B_\nu^{1/2}
\longrightarrow0.
\label{Sub-est-R-92}
\end{align}
The integrability of the weighted momentum used here follows
from \eqref{Thm-assum-2}. Now, we focus on the term $\mc R_{9,3}$. \begin{align}\label{Term-R-73}
\bigg|\int\limits_0^t\mc R_{9,3}(s)\,ds\bigg|
&=
\bigg|
\int\limits_0^t\int\limits_{\Gamma_{c\nu}}
\rho(u-u_{\mc S})\cdot\nabla v^F\cdot u
\,dy\,ds
\bigg|
\notag\\
&=
\bigg|
\int\limits_0^t\int\limits_{\Gamma_{c\nu}}
\frac{\rho(u-u_{\mc S})}{d^2}
\cdot d^2\nabla v^F\cdot u
\,dy\,ds
\bigg| 
\notag\\
&\le
\|d^2\nabla v^F\|_
{L^\infty((0,T)\times\mc F_0)}
\left\|
\frac{\rho(u-u_{\mc S})}{d^2}
\right\|_
{L^2(0,t;L^{6/5}(\Gamma_{c\nu}))}
\notag\\
&\qquad\times
\|u\|_{L^2(0,t;L^6(\Gamma_{c\nu}))}
\notag\\
&\le
C\nu
\left\|
\frac{\rho(u-u_{\mc S})}{d^2}
\right\|_
{L^2(0,t;L^{6/5}(\Gamma_{c\nu}))}
\frac{C}{\sqrt{\nu}}
\notag\\
&\le
C\sqrt{\nu}
\left\|
\frac{\rho(u-u_{\mc S})}{d^2}
\right\|_
{L^2(0,T;L^{6/5}(\Gamma_{c\nu}))} ,
\end{align}
using the uniform bound $\eqref{Uniform-bound-compressible}_1$ and Proposition \ref{Proposition-Fake-Layer}--point \eqref{fake-layer-gradient}.

Combining the estimates \eqref{Term-R-71}--\eqref{Term-R-73} and the assumption \eqref{Thm-assum-2}, we deduce that
\begin{align}
{\int\limits_0^t \mc R_{9}(s)\, ds \longrightarrow 0 \ \text{ as } \ \nu, \veps \to 0. }
\end{align}

Finally, the term $\frac{1}{2}\int\limits_{\mathbb R^3} \rho |v^F|^2 \, dy$ in the expression of $\mc E([\rho, u] \mid [1, U])$ (given by \eqref{Expres-E-1,U}) can be computed in the way we addressed in \eqref{esti-rho-v^F}.

Therefore, by means of  all the above estimates and limiting behaviors,  the Relative energy inequality \eqref{Rel-energy} follows to:
\begin{align}\label{App-Gronwall}
&\frac{1}{2} \int\limits_{\mathbb R^3} \rho |u-u^E|^2 (t)  \, dy + \frac{1}{\veps^2} 
\int\limits_{\mc F_0} \left(H(\rho)-\frac{\partial H}{\partial\rho}(1)(\rho-1)-H(1)\right)(t)\, dy \notag  \\
\leq & 
\frac{1}{2} \int\limits_{\mathbb R^3} \rho_0 \left|\frac{q_0}{\rho_0}-u^E_0\right|^2   \, dy + \frac{1}{\veps^2} 
\int\limits_{\mc F_0} \left(H(\rho_0)-\frac{\partial H}{\partial\rho_0}(1)(\rho_0-1)-H(1)\right)\, dy  \notag \\
& + C \int\limits_0^t \int\limits_{\mathbb R^3} \rho |u-u^E|^2  \,dy\, ds 
+ o(1) ,
\end{align}
where we have used the fact   $\|u-u^E\|^2_{\mc H}=\int\limits_{\mathbb R^3}\rho |u-u^E|^2\, dy$ by definition (see Remark \ref{Remark-detailed-norms}). 

Then, using Gr\"{o}nwall's lemma in \eqref{App-Gronwall},  we have the required convergent result \eqref{main-result-thm-1} w.r.t.  $\veps, \, \nu\to 0$.  
In conclusion, the proof of Theorem \ref{main-thm} is complete.
\end{proof}

\vspace{.1in}
\noindent{\bf Acknowledgments:} K.B. and \v{S}.N. have been supported by the Praemium Academiae of \v{S}.N.  The Institute of Mathematics, CAS is supported by RVO:67985840. A.R is supported by the Grant RYC2022-036183-I funded by MICIU/AEI/10.13039/501100011033 and by ESF+. A.R has been partially supported by the Basque Government through the BERC 2022-2025 program and by the Spanish State Research Agency through BCAM Severo Ochoa CEX2021-001142-S and through project PID2023-146764NB-I00 funded by MICIU/AEI/10.13039/501100011033 and cofunded by the European Union.

%%%%%%%%%%%%%%%%%%%%%%%%%%%%%%%% 

\bibliography{references-singlimits}
\bibliographystyle{plain}

\end{document}